\documentclass{article}
\usepackage[a4paper]{geometry}
\usepackage[utf8]{inputenc}
\usepackage[all]{xy}
\usepackage[initials]{amsrefs}
\usepackage[labelsep=period]{caption}
\usepackage[fleqn]{amsmath}
\usepackage{amsfonts,authblk,bm,doi,enumitem,fancyhdr,float,graphicx,ifthen,amssymb,amsthm,mathrsfs,mathtools,mleftright,relsize,sectsty,subcaption,tikz,tikz-3dplot,titlesec,wrapfig}
\usepackage[svgnames]{xcolor}

\newtheorem{theorem}{Theorem}

\newtheorem{lemma}[theorem]{Lemma}
\newtheorem{corollary}[theorem]{Corollary}
\theoremstyle{definition}
\newtheorem{definition}[theorem]{Definition}
\newtheorem{example}[theorem]{Example}
\newtheorem{remark}[theorem]{Remark}
\theoremstyle{plain}
\newtheorem{innerthm}{Theorem}
\newenvironment{maintheorem}[1]
  {\renewcommand\theinnerthm{#1}\innerthm}
  {\endinnerthm}

\numberwithin{equation}{section} 

\setlist[enumerate]{leftmargin=20pt,itemsep=0pt,topsep=0pt}
\setlist[enumerate,1]{label=\emph{(\roman*)}}
\setlist[itemize]{leftmargin=20pt,itemsep=0pt,topsep=0pt}

\newcommand{\SL}[1]{\ensuremath{\mathrm{SL}_{#1}}}
\newcommand{\Z}{\mathbb{Z}}

\renewcommand{\geq}{\geqslant}
\renewcommand{\leq}{\leqslant}

\usepackage[many]{tcolorbox}

\newcommand{\mytcbox}[2]{
\newtcbox{#1}{
  on line,
  boxrule=0pt,
  arc=2pt,
  colback=#2,
  left=0.3pt,right=0.3pt,top=0.4pt,bottom=0.4pt
}
}

\mytcbox{\frmul}{red!15}
\mytcbox{\frdiv}{black!5}

\usepackage{circuitikz}
\usetikzlibrary{math}

\newcommand{\discfarey}{
   \draw[gray,ultra thin] (0,-1) -- (0,1); 
 
   \foreach \thetaone/\thetatwo/\startangle/\finishangle/\R in 
{
0.0/90.0/-90.0/-180.0/1.0,36.8699/90.0/-53.1301/-180.0/0.5,0.0/36.8699/-90.0/-233.1301/0.333,53.1301/90.0/-36.8699/-180.0/0.333,36.8699/53.1301/-53.1301/-216.8699/0.143,22.6199/36.8699/-67.3801/-233.1301/0.125,0.0/22.6199/-90.0/-247.3801/0.2,61.9275/90.0/-28.0725/-180.0/0.25,53.1301/61.9275/-36.8699/-208.0725/0.077,46.3972/53.1301/-43.6028/-216.8699/0.059,36.8699/46.3972/-53.1301/-223.6028/0.083,28.0725/36.8699/-61.9275/-233.1301/0.077,22.6199/28.0725/-67.3801/-241.9275/0.048,16.2602/22.6199/-73.7398/-247.3801/0.056,0.0/16.2602/-90.0/-253.7398/0.143,67.3801/90.0/-22.6199/-180.0/0.2,61.9275/67.3801/-28.0725/-202.6199/0.048,58.1092/61.9275/-31.8908/-208.0725/0.033,53.1301/58.1092/-36.8699/-211.8908/0.043,48.8879/53.1301/-41.1121/-216.8699/0.037,46.3972/48.8879/-43.6028/-221.1121/0.022,43.6028/46.3972/-46.3972/-223.6028/0.024,36.8699/43.6028/-53.1301/-226.3972/0.059,30.5102/36.8699/-59.4898/-233.1301/0.056,28.0725/30.5102/-61.9275/-239.4898/0.021,25.9892/28.0725/-64.0108/-241.9275/0.018,22.6199/25.9892/-67.3801/-244.0108/0.029,18.9246/22.6199/-71.0754/-247.3801/0.032,16.2602/18.9246/-73.7398/-251.0754/0.023,12.6804/16.2602/-77.3196/-253.7398/0.031,0.0/12.6804/-90.0/-257.3196/0.111,71.0754/90.0/-18.9246/-180.0/0.167,67.3801/71.0754/-22.6199/-198.9246/0.032,64.9424/67.3801/-25.0576/-202.6199/0.021,61.9275/64.9424/-28.0725/-205.0576/0.026,59.4898/61.9275/-30.5102/-208.0725/0.021,58.1092/59.4898/-31.8908/-210.5102/0.012,56.6015/58.1092/-33.3985/-211.8908/0.013,53.1301/56.6015/-36.8699/-213.3985/0.03,50.0338/53.1301/-39.9662/-216.8699/0.027,48.8879/50.0338/-41.1121/-219.9662/0.01,47.925/48.8879/-42.075/-221.1121/0.008,46.3972/47.925/-43.6028/-222.075/0.013,44.7603/46.3972/-45.2397/-223.6028/0.014,43.6028/44.7603/-46.3972/-225.2397/0.01,42.075/43.6028/-47.925/-226.3972/0.013,36.8699/42.075/-53.1301/-227.925/0.045,31.8908/36.8699/-58.1092/-233.1301/0.043,30.5102/31.8908/-59.4898/-238.1092/0.012,29.4871/30.5102/-60.5129/-239.4898/0.009,28.0725/29.4871/-61.9275/-240.5129/0.012,26.785/28.0725/-63.215/-241.9275/0.011,25.9892/26.785/-64.0108/-243.215/0.007,25.0576/25.9892/-64.9424/-244.0108/0.008,22.6199/25.0576/-67.3801/-244.9424/0.021,20.016/22.6199/-69.984/-247.3801/0.023,18.9246/20.016/-71.0754/-249.984/0.01,17.9453/18.9246/-72.0547/-251.0754/0.009,16.2602/17.9453/-73.7398/-252.0547/0.015,14.25/16.2602/-75.75/-253.7398/0.018,12.6804/14.25/-77.3196/-255.75/0.014,10.3889/12.6804/-79.6111/-257.3196/0.02,0.0/10.3889/-90.0/-259.6111/0.091,73.7398/90.0/-16.2602/-180.0/0.143,71.0754/73.7398/-18.9246/-196.2602/0.023,69.3903/71.0754/-20.6097/-198.9246/0.015,67.3801/69.3903/-22.6199/-200.6097/0.018,65.8105/67.3801/-24.1895/-202.6199/0.014,64.9424/65.8105/-25.0576/-204.1895/0.008,64.0108/64.9424/-25.9892/-205.0576/0.008,61.9275/64.0108/-28.0725/-205.9892/0.018,60.1372/61.9275/-29.8628/-208.0725/0.016,59.4898/60.1372/-30.5102/-209.8628/0.006,58.9518/59.4898/-31.0482/-210.5102/0.005,58.1092/58.9518/-31.8908/-211.0482/0.007,57.2209/58.1092/-32.7791/-211.8908/0.008,56.6015/57.2209/-33.3985/-212.7791/0.005,55.7945/56.6015/-34.2055/-213.3985/0.007,53.1301/55.7945/-36.8699/-214.2055/0.023,50.6924/53.1301/-39.3076/-216.8699/0.021,50.0338/50.6924/-39.9662/-219.3076/0.006,49.5503/50.0338/-40.4497/-219.9662/0.004,48.8879/49.5503/-41.1121/-220.4497/0.006,48.2911/48.8879/-41.7089/-221.1121/0.005,47.925/48.2911/-42.075/-221.7089/0.003,47.499/47.925/-42.501/-222.075/0.004,46.3972/47.499/-43.6028/-222.501/0.01,45.2397/46.3972/-44.7603/-223.6028/0.01,44.7603/45.2397/-45.2397/-224.7603/0.004,44.3327/44.7603/-45.6673/-225.2397/0.004,43.6028/44.3327/-46.3972/-225.6673/0.006,42.7412/43.6028/-47.2588/-226.3972/0.008,42.075/42.7412/-47.925/-227.2588/0.006,41.1121/42.075/-48.8879/-227.925/0.008,36.8699/41.1121/-53.1301/-228.8879/0.037,32.7791/36.8699/-57.2209/-233.1301/0.036,31.8908/32.7791/-58.1092/-237.2209/0.008,31.2845/31.8908/-58.7155/-238.1092/0.005,30.5102/31.2845/-59.4898/-238.7155/0.007,29.8628/30.5102/-60.1372/-239.4898/0.006,29.4871/29.8628/-60.5129/-240.1372/0.003,29.0689/29.4871/-60.9311/-240.5129/0.004,28.0725/29.0689/-61.9275/-240.9311/0.009,27.1409/28.0725/-62.8591/-241.9275/0.008,26.785/27.1409/-63.215/-242.8591/0.003,26.481/26.785/-63.519/-243.215/0.003,25.9892/26.481/-64.0108/-243.519/0.004,25.4487/25.9892/-64.5513/-244.0108/0.005,25.0576/25.4487/-64.9424/-244.5513/0.003,24.5295/25.0576/-65.4705/-244.9424/0.005,22.6199/24.5295/-67.3801/-245.4705/0.017,20.6097/22.6199/-69.3903/-247.3801/0.018,20.016/20.6097/-69.984/-249.3903/0.005,19.5648/20.016/-70.4352/-249.984/0.004,18.9246/19.5648/-71.0754/-250.4352/0.006,18.3247/18.9246/-71.6753/-251.0754/0.005,17.9453/18.3247/-72.0547/-251.6753/0.003,17.4923/17.9453/-72.5077/-252.0547/0.004,16.2602/17.4923/-73.7398/-252.5077/0.011,14.8628/16.2602/-75.1372/-253.7398/0.012,14.25/14.8628/-75.75/-255.1372/0.005,13.6855/14.25/-76.3145/-255.75/0.005,12.6804/13.6855/-77.3196/-256.3145/0.009,11.4212/12.6804/-78.5788/-257.3196/0.011,10.3889/11.4212/-79.6111/-258.5788/0.009,8.7974/10.3889/-81.2026/-259.6111/0.014,0.0/8.7974/-90.0/-261.2026/0.077,75.75/90.0/-14.25/-180.0/0.125,73.7398/75.75/-16.2602/-194.25/0.018,72.5077/73.7398/-17.4923/-196.2602/0.011,71.0754/72.5077/-18.9246/-197.4923/0.012,69.984/71.0754/-20.016/-198.9246/0.01,69.3903/69.984/-20.6097/-200.016/0.005,68.7607/69.3903/-21.2393/-200.6097/0.005,67.3801/68.7607/-22.6199/-201.2393/0.012,66.2227/67.3801/-23.7773/-202.6199/0.01,65.8105/66.2227/-24.1895/-203.7773/0.004,65.4705/65.8105/-24.5295/-204.1895/0.003,64.9424/65.4705/-25.0576/-204.5295/0.005,64.3915/64.9424/-25.6085/-205.0576/0.005,64.0108/64.3915/-25.9892/-205.6085/0.003,63.519/64.0108/-26.481/-205.9892/0.004,61.9275/63.519/-28.0725/-206.481/0.014,60.5129/61.9275/-29.4871/-208.0725/0.012,60.1372/60.5129/-29.8628/-209.4871/0.003,58.7155/58.9518/-31.2845/-211.0482/0.002,58.1092/58.7155/-31.8908/-211.2845/0.005,57.4796/58.1092/-32.5204/-211.8908/0.005,57.2209/57.4796/-32.7791/-212.5204/0.002,56.145/56.6015/-33.855/-213.3985/0.004,55.7945/56.145/-34.2055/-213.855/0.003,55.292/55.7945/-34.708/-214.2055/0.004,53.1301/55.292/-36.8699/-214.708/0.019,51.1199/53.1301/-38.8801/-216.8699/0.018,50.6924/51.1199/-39.3076/-218.8801/0.004,47.2588/47.499/-42.7412/-222.501/0.002,46.3972/47.2588/-43.6028/-222.7412/0.008,45.502/46.3972/-44.498/-223.6028/0.008,45.2397/45.502/-44.7603/-224.498/0.002,44.1358/44.3327/-45.8642/-225.6673/0.002,43.6028/44.1358/-46.3972/-225.8642/0.005,43.0029/43.6028/-46.9971/-226.3972/0.005,42.7412/43.0029/-47.2588/-226.9971/0.002,41.5445/42.075/-48.4555/-227.925/0.005,41.1121/41.5445/-48.8879/-228.4555/0.004,40.4497/41.1121/-49.5503/-228.8879/0.006,36.8699/40.4497/-53.1301/-229.5503/0.031,33.3985/36.8699/-56.6015/-233.1301/0.03,32.7791/33.3985/-57.2209/-236.6015/0.005,32.3784/32.7791/-57.6216/-237.2209/0.003,31.8908/32.3784/-58.1092/-237.6216/0.004,31.0482/31.2845/-58.9518/-238.7155/0.002,30.5102/31.0482/-59.4898/-238.9518/0.005,28.8415/29.0689/-61.1585/-240.9311/0.002,28.0725/28.8415/-61.9275/-241.1585/0.007,27.3426/28.0725/-62.6574/-241.9275/0.006,27.1409/27.3426/-62.8591/-242.6574/0.002,24.1895/24.5295/-65.8105/-245.4705/0.003,22.6199/24.1895/-67.3801/-245.8105/0.014,20.983/22.6199/-69.017/-247.3801/0.014,20.6097/20.983/-69.3903/-249.017/0.003,17.2313/17.4923/-72.7687/-252.5077/0.002,16.2602/17.2313/-73.7398/-252.7687/0.008,15.1893/16.2602/-74.8107/-253.7398/0.009,14.8628/15.1893/-75.1372/-254.8107/0.003,13.4197/13.6855/-76.5803/-256.3145/0.002,12.6804/13.4197/-77.3196/-256.5803/0.006,11.8123/12.6804/-78.1877/-257.3196/0.008,11.4212/11.8123/-78.5788/-258.1877/0.003,11.0551/11.4212/-78.9449/-258.5788/0.003,10.3889/11.0551/-79.6111/-258.9449/0.006,9.5273/10.3889/-80.4727/-259.6111/0.008,8.7974/9.5273/-81.2026/-260.4727/0.006,7.6281/8.7974/-82.3719/-261.2026/0.01,0.0/7.6281/-90.0/-262.3719/0.067,77.3196/90.0/-12.6804/-180.0/0.111,75.75/77.3196/-14.25/-192.6804/0.014,74.8107/75.75/-15.1893/-194.25/0.008,73.7398/74.8107/-16.2602/-195.1893/0.009,72.9385/73.7398/-17.0615/-196.2602/0.007,72.5077/72.9385/-17.4923/-197.0615/0.004,72.0547/72.5077/-17.9453/-197.4923/0.004,71.0754/72.0547/-18.9246/-197.9453/0.009,70.2684/71.0754/-19.7316/-198.9246/0.007,69.984/70.2684/-20.016/-199.7316/0.002,68.4314/68.7607/-21.5686/-201.2393/0.003,67.3801/68.4314/-22.6199/-201.5686/0.009,66.4634/67.3801/-23.5366/-202.6199/0.008,66.2227/66.4634/-23.7773/-203.5366/0.002,63.215/63.519/-26.785/-206.481/0.003,61.9275/63.215/-28.0725/-206.785/0.011,60.7583/61.9275/-29.2417/-208.0725/0.01,60.5129/60.7583/-29.4871/-209.2417/0.002,54.9489/55.292/-35.0511/-214.708/0.003,53.1301/54.9489/-36.8699/-215.0511/0.016,51.4199/53.1301/-38.5801/-216.8699/0.015,51.1199/51.4199/-38.8801/-218.5801/0.003,47.1045/47.2588/-42.8955/-222.7412/0.001,46.3972/47.1045/-43.6028/-222.8955/0.006,45.6673/46.3972/-44.3327/-223.6028/0.006,45.502/45.6673/-44.498/-224.3327/0.001,39.9662/40.4497/-50.0338/-229.5503/0.004,36.8699/39.9662/-53.1301/-230.0338/0.027,33.855/36.8699/-56.145/-233.1301/0.026,33.3985/33.855/-56.6015/-236.145/0.004,28.6987/28.8415/-61.3013/-241.1585/0.001,28.0725/28.6987/-61.9275/-241.3013/0.005,27.4725/28.0725/-62.5275/-241.9275/0.005,27.3426/27.4725/-62.6574/-242.5275/0.001,23.9523/24.1895/-66.0477/-245.8105/0.002,22.6199/23.9523/-67.3801/-246.0477/0.012,21.2393/22.6199/-68.7607/-247.3801/0.012,20.983/21.2393/-69.017/-248.7607/0.002,17.0615/17.2313/-72.9385/-252.7687/0.001,16.2602/17.0615/-73.7398/-252.9385/0.007,15.3921/16.2602/-74.6079/-253.7398/0.008,15.1893/15.3921/-74.8107/-254.6079/0.002,13.265/13.4197/-76.735/-256.5803/0.001,12.6804/13.265/-77.3196/-256.735/0.005,12.018/12.6804/-77.982/-257.3196/0.006,11.8123/12.018/-78.1877/-257.982/0.002,9.7982/10.3889/-80.2018/-259.6111/0.005,9.5273/9.7982/-80.4727/-260.2018/0.002,9.2709/9.5273/-80.7291/-260.4727/0.002,8.7974/9.2709/-81.2026/-260.7291/0.004,8.1712/8.7974/-81.8288/-261.2026/0.005,7.6281/8.1712/-82.3719/-261.8288/0.005,6.7329/7.6281/-83.2671/-262.3719/0.008,0.0/6.7329/-90.0/-263.2671/0.059,78.5788/90.0/-11.4212/-180.0/0.1,77.3196/78.5788/-12.6804/-191.4212/0.011,76.5803/77.3196/-13.4197/-192.6804/0.006,75.75/76.5803/-14.25/-193.4197/0.007,75.1372/75.75/-14.8628/-194.25/0.005,74.8107/75.1372/-15.1893/-194.8628/0.003,74.4697/74.8107/-15.5303/-195.1893/0.003,73.7398/74.4697/-16.2602/-195.5303/0.006,73.1461/73.7398/-16.8539/-196.2602/0.005,72.9385/73.1461/-17.0615/-196.8539/0.002,71.8194/72.0547/-18.1806/-197.9453/0.002,71.0754/71.8194/-18.9246/-198.1806/0.006,70.4352/71.0754/-19.5648/-198.9246/0.006,70.2684/70.4352/-19.7316/-199.5648/0.001,68.2289/68.4314/-21.7711/-201.5686/0.002,67.3801/68.2289/-22.6199/-201.7711/0.007,66.6213/67.3801/-23.3787/-202.6199/0.007,66.4634/66.6213/-23.5366/-203.3787/0.001,63.0085/63.215/-26.9915/-206.785/0.002,61.9275/63.0085/-28.0725/-206.9915/0.009,60.9311/61.9275/-29.0689/-208.0725/0.009,60.7583/60.9311/-29.2417/-209.0689/0.002,54.6998/54.9489/-35.3002/-215.0511/0.002,53.1301/54.6998/-36.8699/-215.3002/0.014,51.642/53.1301/-38.358/-216.8699/0.013,51.4199/51.642/-38.5801/-218.358/0.002,46.9971/47.1045/-43.0029/-222.8955/0.001,46.3972/46.9971/-43.6028/-223.0029/0.005,45.7811/46.3972/-44.2189/-223.6028/0.005,45.6673/45.7811/-44.3327/-224.2189/0.001,39.5978/39.9662/-50.4022/-230.0338/0.003,36.8699/39.5978/-53.1301/-230.4022/0.024,34.2055/36.8699/-55.7945/-233.1301/0.023,33.855/34.2055/-56.145/-235.7945/0.003,23.7773/23.9523/-66.2227/-246.0477/0.002,22.6199/23.7773/-67.3801/-246.2227/0.01,21.4262/22.6199/-68.5738/-247.3801/0.01,21.2393/21.4262/-68.7607/-248.5738/0.002,16.9423/17.0615/-73.0577/-252.9385/0.001,16.2602/16.9423/-73.7398/-253.0577/0.006,15.5303/16.2602/-74.4697/-253.7398/0.006,15.3921/15.5303/-74.6079/-254.4697/0.001,7.1527/7.6281/-82.8473/-262.3719/0.004,6.7329/7.1527/-83.2671/-262.8473/0.004,6.0256/6.7329/-83.9744/-263.2671/0.006,0.0/6.0256/-90.0/-263.9744/0.053,79.6111/90.0/-10.3889/-180.0/0.091,78.5788/79.6111/-11.4212/-190.3889/0.009,77.982/78.5788/-12.018/-191.4212/0.005,77.3196/77.982/-12.6804/-192.018/0.006,76.8361/77.3196/-13.1639/-192.6804/0.004,76.5803/76.8361/-13.4197/-193.1639/0.002,76.3145/76.5803/-13.6855/-193.4197/0.002,75.75/76.3145/-14.25/-193.6855/0.005,74.2934/74.4697/-15.7066/-195.5303/0.002,73.7398/74.2934/-16.2602/-195.7066/0.005,71.6753/71.8194/-18.3247/-198.1806/0.001,71.0754/71.6753/-18.9246/-198.3247/0.005,68.0919/68.2289/-21.9081/-201.7711/0.001,67.3801/68.0919/-22.6199/-201.9081/0.006,66.7327/67.3801/-23.2673/-202.6199/0.006,66.6213/66.7327/-23.3787/-203.2673/0.001,62.8591/63.0085/-27.1409/-206.9915/0.001,61.9275/62.8591/-28.0725/-207.1409/0.008,61.0594/61.9275/-28.9406/-208.0725/0.008,60.9311/61.0594/-29.0689/-208.9406/0.001,54.5107/54.6998/-35.4893/-215.3002/0.002,53.1301/54.5107/-36.8699/-215.4893/0.012,51.813/53.1301/-38.187/-216.8699/0.011,51.642/51.813/-38.358/-218.187/0.001,39.3076/39.5978/-50.6924/-230.4022/0.003,36.8699/39.3076/-53.1301/-230.6924/0.021,34.4829/36.8699/-55.5171/-233.1301/0.021,34.2055/34.4829/-55.7945/-235.5171/0.002,23.643/23.7773/-66.357/-246.2227/0.001,22.6199/23.643/-67.3801/-246.357/0.009,21.5686/22.6199/-68.4314/-247.3801/0.009,21.4262/21.5686/-68.5738/-248.4314/0.001,15.6306/16.2602/-74.3694/-253.7398/0.005,15.5303/15.6306/-74.4697/-254.3694/0.001,6.3597/6.7329/-83.6403/-263.2671/0.003,6.0256/6.3597/-83.9744/-263.6403/0.003,5.4526/6.0256/-84.5474/-263.9744/0.005,0.0/5.4526/-90.0/-264.5474/0.048,80.4727/90.0/-9.5273/-180.0/0.083,79.6111/80.4727/-10.3889/-189.5273/0.008,79.1193/79.6111/-10.8807/-190.3889/0.004,78.5788/79.1193/-11.4212/-190.8807/0.005,67.9929/68.0919/-22.0071/-201.9081/0.001,67.3801/67.9929/-22.6199/-202.0071/0.005,62.746/62.8591/-27.254/-207.1409/0.001,61.9275/62.746/-28.0725/-207.254/0.007,61.1585/61.9275/-28.8415/-208.0725/0.007,61.0594/61.1585/-28.9406/-208.8415/0.001,54.3622/54.5107/-35.6378/-215.4893/0.001,53.1301/54.3622/-36.8699/-215.6378/0.011,51.9488/53.1301/-38.0512/-216.8699/0.01,51.813/51.9488/-38.187/-218.0512/0.001,39.0733/39.3076/-50.9267/-230.6924/0.002,36.8699/39.0733/-53.1301/-230.9267/0.019,34.708/36.8699/-55.292/-233.1301/0.019,34.4829/34.708/-55.5171/-235.292/0.002,23.5366/23.643/-66.4634/-246.357/0.001,22.6199/23.5366/-67.3801/-246.4634/0.008,21.6806/22.6199/-68.3194/-247.3801/0.008,21.5686/21.6806/-68.4314/-248.3194/0.001,4.9791/5.4526/-85.0209/-264.5474/0.004,0.0/4.9791/-90.0/-265.0209/0.043,81.2026/90.0/-8.7974/-180.0/0.077,80.4727/81.2026/-9.5273/-188.7974/0.006,80.0605/80.4727/-9.9395/-189.5273/0.004,79.6111/80.0605/-10.3889/-189.9395/0.004,62.6574/62.746/-27.3426/-207.254/0.001,61.9275/62.6574/-28.0725/-207.3426/0.006,61.2372/61.9275/-28.7628/-208.0725/0.006,61.1585/61.2372/-28.8415/-208.7628/0.001,54.2426/54.3622/-35.7574/-215.6378/0.001,53.1301/54.2426/-36.8699/-215.7574/0.01,52.0592/53.1301/-37.9408/-216.8699/0.009,51.9488/52.0592/-38.0512/-217.9408/0.001,38.8801/39.0733/-51.1199/-230.9267/0.002,36.8699/38.8801/-53.1301/-231.1199/0.018,34.8944/36.8699/-55.1056/-233.1301/0.017,34.708/34.8944/-55.292/-235.1056/0.002,23.4502/23.5366/-66.5498/-246.4634/0.001,22.6199/23.4502/-67.3801/-246.5498/0.007,21.7711/22.6199/-68.2289/-247.3801/0.007,21.6806/21.7711/-68.3194/-248.2289/0.001,4.5812/4.9791/-85.4188/-265.0209/0.003,0.0/4.5812/-90.0/-265.4188/0.04,81.8288/90.0/-8.1712/-180.0/0.071,81.2026/81.8288/-8.7974/-188.1712/0.005,80.8522/81.2026/-9.1478/-188.7974/0.003,80.4727/80.8522/-9.5273/-189.1478/0.003,62.5861/62.6574/-27.4139/-207.3426/0.001,61.9275/62.5861/-28.0725/-207.4139/0.006,61.3013/61.9275/-28.6987/-208.0725/0.005,61.2372/61.3013/-28.7628/-208.6987/0.001,54.1442/54.2426/-35.8558/-215.7574/0.001,53.1301/54.1442/-36.8699/-215.8558/0.009,52.1507/53.1301/-37.8493/-216.8699/0.009,52.0592/52.1507/-37.9408/-217.8493/0.001,38.718/38.8801/-51.282/-231.1199/0.001,36.8699/38.718/-53.1301/-231.282/0.016,35.0511/36.8699/-54.9489/-233.1301/0.016,34.8944/35.0511/-55.1056/-234.9489/0.001,23.3787/23.4502/-66.6213/-246.5498/0.001,22.6199/23.3787/-67.3801/-246.6213/0.007,21.8456/22.6199/-68.1544/-247.3801/0.007,21.7711/21.8456/-68.2289/-248.1544/0.001,4.2422/4.5812/-85.7578/-265.4188/0.003,0.0/4.2422/-90.0/-265.7578/0.037,82.3719/90.0/-7.6281/-180.0/0.067,81.8288/82.3719/-8.1712/-187.6281/0.005,54.0617/54.1442/-35.9383/-215.8558/0.001,53.1301/54.0617/-36.8699/-215.9383/0.008
}
    {
\pgfmathsetmacro{\costheta}{cos(\thetaone)}
\pgfmathsetmacro{\sintheta}{sin(\thetaone)}
	\draw[gray,ultra thin](\costheta,\sintheta) arc (\startangle:\finishangle:\R); 
     	\draw[gray,ultra thin](-\costheta,\sintheta) arc (180-\startangle:180-\finishangle:\R); 
	\draw[gray,ultra thin](\costheta,-\sintheta) arc (-\startangle:-\finishangle:\R); 
	\draw[gray,ultra thin](-\costheta,-\sintheta) arc (\startangle+180:\finishangle+180:\R); 
   }
   \draw (0,0) circle (1); 
}

\def\dischorocyclenolabel[#1](#2:#3)
{
	\tikzmath{
	    \A=#2;
      	    \B=#3;
      	    \R = 1/(1+\A*\A+\B*\B);
      	    if  #3!=0 then {
      	       \thetaone = 2*atan(\A/\B))-90;
          } else {
             \thetaone = 90;
          };
	} 
	\draw[#1] (\thetaone:{1-\R}) circle (\R);
}  

\def\dischorocycle[#1](#2:#3)
{
	\dischorocyclenolabel[#1](#2:#3);
	\pgfmathsetmacro{\radius}{1.11-0.03*cos(2*\thetaone)}
	\node at (\thetaone:\radius) {$\frac{\A}{\B}$};
}  

\def\shline[#1](#2:#3:#4:#5){%
  \begingroup

  \pgfmathtruncatemacro{\anti}{((#2)*(#4)+(#3)*(#5)==0) ? 1 : 0}

  \pgfmathsetmacro{\xA}{2*(#2)*(#3)/((#2)*(#2)+(#3)*(#3))}
  \pgfmathsetmacro{\yA}{((#2)*(#2)-(#3)*(#3))/((#2)*(#2)+(#3)*(#3))}
  \pgfmathsetmacro{\xB}{2*(#4)*(#5)/((#4)*(#4)+(#5)*(#5))}
  \pgfmathsetmacro{\yB}{((#4)*(#4)-(#5)*(#5))/((#4)*(#4)+(#5)*(#5))}

  \ifnum\anti=1
    \draw[#1] (\xB,\yB)--(\xA,\yA) node[currarrow,pos = 0.55,xscale=-1,sloped,scale=0.3]{};
  \else
    \pgfmathsetmacro{\dotp}{\xA*\xB+\yA*\yB}
    \pgfmathsetmacro{\den}{1+\dotp}
    \pgfmathsetmacro{\cx}{(\xA+\xB)/\den}
    \pgfmathsetmacro{\cy}{(\yA+\yB)/\den}
    \pgfmathsetmacro{\rad}{sqrt((1-\dotp)/\den)}
    \pgfmathsetmacro{\angA}{atan2(\yA-\cy,\xA-\cx)}
    \pgfmathsetmacro{\angB}{atan2(\yB-\cy,\xB-\cx)}

    \pgfmathsetmacro{\diff}{\angB-\angA}
    \ifdim\diff pt > 180pt\relax
      \pgfmathsetmacro{\angB}{\angB-360}
    \fi
    \ifdim\diff pt < -180pt\relax
      \pgfmathsetmacro{\angB}{\angB+360}
    \fi

    \draw[#1] (\xA,\yA)
      arc[start angle=\angA,end angle=\angB,radius=\rad] node[currarrow,pos = 0.5,allow upside down,sloped,scale=0.3]{};
  \fi

  \endgroup
}

\newcommand{\fcRightAngleSize}{0.035}
\newcommand{\fcDiameterTolerance}{0.0005}
\tikzset{
  fc right angle/.style={black,line width=0.25pt},
  fc segment label/.style={above,inner sep=1pt}
}

\newcommand{\fcRightAngleMark}[7]{
  \pgfmathsetmacro{\fcrax}{#1}%
  \pgfmathsetmacro{\fcray}{#2}%
  \pgfmathsetmacro{\fcraux}{#3}%
  \pgfmathsetmacro{\fcrauy}{#4}%
  \pgfmathsetmacro{\fcravx}{#5}%
  \pgfmathsetmacro{\fcravy}{#6}%
  \pgfmathsetmacro{\fcras}{#7}%
  \pgfmathsetmacro{\fcralenone}{sqrt(max(\fcraux*\fcraux+\fcrauy*\fcrauy,0.000001))}%
  \pgfmathsetmacro{\fcralentwo}{sqrt(max(\fcravx*\fcravx+\fcravy*\fcravy,0.000001))}%
  \pgfmathsetmacro{\fcrauux}{\fcraux/\fcralenone}%
  \pgfmathsetmacro{\fcrauuy}{\fcrauy/\fcralenone}%
  \pgfmathsetmacro{\fcravvx}{\fcravx/\fcralentwo}%
  \pgfmathsetmacro{\fcravvy}{\fcravy/\fcralentwo}%
  \draw[fc right angle]
    ({\fcrax+\fcras*\fcrauux},{\fcray+\fcras*\fcrauuy})
    -- ({\fcrax+\fcras*(\fcrauux+\fcravvx)},{\fcray+\fcras*(\fcrauuy+\fcravvy)})
    -- ({\fcrax+\fcras*\fcravvx},{\fcray+\fcras*\fcravvy});%
}

\newcommand{\fordsegment}[3][]{%
  \begingroup
  \def\fcparseford##1:##2:##3:##4\relax{%
    \pgfmathsetmacro{\fcfa}{##1}%
    \pgfmathsetmacro{\fcfb}{##2}%
    \pgfmathsetmacro{\fcfc}{##3}%
    \pgfmathsetmacro{\fcfd}{##4}%
  }%
  \expandafter\fcparseford#2\relax
  \pgfmathsetmacro{\fcdena}{\fcfa*\fcfa+\fcfb*\fcfb}%
  \pgfmathsetmacro{\fcdenb}{\fcfc*\fcfc+\fcfd*\fcfd}%
  \pgfmathsetmacro{\fcax}{2*\fcfa*\fcfb/\fcdena}%
  \pgfmathsetmacro{\fcay}{(\fcfa*\fcfa-\fcfb*\fcfb)/\fcdena}%
  \pgfmathsetmacro{\fcbx}{2*\fcfc*\fcfd/\fcdenb}%
  \pgfmathsetmacro{\fcby}{(\fcfc*\fcfc-\fcfd*\fcfd)/\fcdenb}%
  \pgfmathsetmacro{\fcra}{1/(1+\fcfa*\fcfa+\fcfb*\fcfb)}%
  \pgfmathsetmacro{\fcrb}{1/(1+\fcfc*\fcfc+\fcfd*\fcfd)}%
  \pgfmathsetmacro{\fchax}{(1-\fcra)*\fcax}%
  \pgfmathsetmacro{\fchay}{(1-\fcra)*\fcay}%
  \pgfmathsetmacro{\fchbx}{(1-\fcrb)*\fcbx}%
  \pgfmathsetmacro{\fchby}{(1-\fcrb)*\fcby}%
  \pgfmathtruncatemacro{\fcanti}{((\fcfa)*(\fcfc)+(\fcfb)*(\fcfd)==0) ? 1 : 0}%
  \ifnum\fcanti=1
    \pgfmathsetmacro{\fcix}{(1-2*\fcra)*\fcax}%
    \pgfmathsetmacro{\fciy}{(1-2*\fcra)*\fcay}%
    \pgfmathsetmacro{\fcjx}{(1-2*\fcrb)*\fcbx}%
    \pgfmathsetmacro{\fcjy}{(1-2*\fcrb)*\fcby}%
    \draw[#1] (\fcix,\fciy)--(\fcjx,\fcjy);%
    \pgfmathsetmacro{\fclx}{(\fcix+\fcjx)/2}%
    \pgfmathsetmacro{\fcly}{(\fciy+\fcjy)/2}%
    \node[fc segment label] at (\fclx,\fcly) {#3};%
    \pgfmathsetmacro{\fcuax}{\fcix-\fchax}%
    \pgfmathsetmacro{\fcuay}{\fciy-\fchay}%
    \pgfmathsetmacro{\fcvax}{-\fcay}%
    \pgfmathsetmacro{\fcvay}{\fcax}%
    \fcRightAngleMark{\fcix}{\fciy}{\fcuax}{\fcuay}{\fcvax}{\fcvay}{\fcRightAngleSize}%
    \pgfmathsetmacro{\fcubx}{\fcjx-\fchbx}%
    \pgfmathsetmacro{\fcuby}{\fcjy-\fchby}%
    \pgfmathsetmacro{\fcvbx}{-\fcby}%
    \pgfmathsetmacro{\fcvby}{\fcbx}%
    \fcRightAngleMark{\fcjx}{\fcjy}{\fcubx}{\fcuby}{\fcvbx}{\fcvby}{\fcRightAngleSize}%
  \else
    \pgfmathsetmacro{\fcdot}{\fcax*\fcbx+\fcay*\fcby}%
    \pgfmathsetmacro{\fcgden}{1+\fcdot}%
    \pgfmathsetmacro{\fccx}{(\fcax+\fcbx)/\fcgden}%
    \pgfmathsetmacro{\fccy}{(\fcay+\fcby)/\fcgden}%
    \pgfmathsetmacro{\fcgr}{sqrt(max((1-\fcdot)/\fcgden,0))}%
    \pgfmathsetmacro{\fceax}{-\fcay}%
    \pgfmathsetmacro{\fceay}{\fcax}%
    \pgfmathsetmacro{\fcka}{\fccx*\fceax+\fccy*\fceay}%
    \pgfmathsetmacro{\fcqa}{\fcka*\fcka+\fcra*\fcra}%
    \pgfmathsetmacro{\fcta}{-2*\fcra*\fcka*\fcka/\fcqa}%
    \pgfmathsetmacro{\fcsa}{2*\fcra*\fcra*\fcka/\fcqa}%
    \pgfmathsetmacro{\fcix}{(1+\fcta)*\fcax+\fcsa*\fceax}%
    \pgfmathsetmacro{\fciy}{(1+\fcta)*\fcay+\fcsa*\fceay}%
    \pgfmathsetmacro{\fcebx}{-\fcby}%
    \pgfmathsetmacro{\fceby}{\fcbx}%
    \pgfmathsetmacro{\fckb}{\fccx*\fcebx+\fccy*\fceby}%
    \pgfmathsetmacro{\fcqb}{\fckb*\fckb+\fcrb*\fcrb}%
    \pgfmathsetmacro{\fctb}{-2*\fcrb*\fckb*\fckb/\fcqb}%
    \pgfmathsetmacro{\fcsb}{2*\fcrb*\fcrb*\fckb/\fcqb}%
    \pgfmathsetmacro{\fcjx}{(1+\fctb)*\fcbx+\fcsb*\fcebx}%
    \pgfmathsetmacro{\fcjy}{(1+\fctb)*\fcby+\fcsb*\fceby}%
    \pgfmathsetmacro{\fcanga}{atan2(\fciy-\fccy,\fcix-\fccx)}%
    \pgfmathsetmacro{\fcangb}{atan2(\fcjy-\fccy,\fcjx-\fccx)}%
    \pgfmathsetmacro{\fcdiff}{\fcangb-\fcanga}%
    \ifdim\fcdiff pt > 180pt\relax
      \pgfmathsetmacro{\fcangb}{\fcangb-360}%
    \fi
    \ifdim\fcdiff pt < -180pt\relax
      \pgfmathsetmacro{\fcangb}{\fcangb+360}%
    \fi
    \draw[#1] (\fcix,\fciy) arc[start angle=\fcanga,end angle=\fcangb,radius=\fcgr];%
    \pgfmathsetmacro{\fcmidang}{(\fcanga+\fcangb)/2}%
    \pgfmathsetmacro{\fclx}{\fccx+\fcgr*cos(\fcmidang)}%
    \pgfmathsetmacro{\fcly}{\fccy+\fcgr*sin(\fcmidang)}%
    \node[fc segment label,left] at (\fclx,\fcly) {#3};%
    \pgfmathsetmacro{\fcuax}{\fcix-\fchax}%
    \pgfmathsetmacro{\fcuay}{\fciy-\fchay}%
    \pgfmathsetmacro{\fcvax}{\fcix-\fccx}%
    \pgfmathsetmacro{\fcvay}{\fciy-\fccy}%
    \fcRightAngleMark{\fcix}{\fciy}{\fcuax}{\fcuay}{\fcvax}{\fcvay}{\fcRightAngleSize}%
    \pgfmathsetmacro{\fcubx}{\fcjx-\fchbx}%
    \pgfmathsetmacro{\fcuby}{\fcjy-\fchby}%
    \pgfmathsetmacro{\fcvbx}{\fcjx-\fccx}%
    \pgfmathsetmacro{\fcvby}{\fcjy-\fccy}%
    \fcRightAngleMark{\fcjx}{\fcjy}{\fcubx}{\fcuby}{\fcvbx}{\fcvby}{\fcRightAngleSize}%
  \fi
  \endgroup
}

\newcommand{\fareysegment}[4][]{%
  \begingroup
  \def\fcparseone##1:##2:##3:##4\relax{%
    \pgfmathsetmacro{\fcoa}{##1}%
    \pgfmathsetmacro{\fcob}{##2}%
    \pgfmathsetmacro{\fcoc}{##3}%
    \pgfmathsetmacro{\fcod}{##4}%
  }%
  \def\fcparsetwo##1:##2:##3:##4\relax{%
    \pgfmathsetmacro{\fcta}{##1}%
    \pgfmathsetmacro{\fctb}{##2}%
    \pgfmathsetmacro{\fctc}{##3}%
    \pgfmathsetmacro{\fctd}{##4}%
  }%
  \expandafter\fcparseone#2\relax
  \expandafter\fcparsetwo#3\relax
  \pgfmathsetmacro{\fcodenone}{\fcoa*\fcoa+\fcob*\fcob}%
  \pgfmathsetmacro{\fcodentwo}{\fcoc*\fcoc+\fcod*\fcod}%
  \pgfmathsetmacro{\fcoax}{2*\fcoa*\fcob/\fcodenone}%
  \pgfmathsetmacro{\fcoay}{(\fcoa*\fcoa-\fcob*\fcob)/\fcodenone}%
  \pgfmathsetmacro{\fcobx}{2*\fcoc*\fcod/\fcodentwo}%
  \pgfmathsetmacro{\fcoby}{(\fcoc*\fcoc-\fcod*\fcod)/\fcodentwo}%
  \pgfmathsetmacro{\fcodot}{\fcoax*\fcobx+\fcoay*\fcoby}%
  \pgfmathsetmacro{\fcoantiabs}{abs(1+\fcodot)}%
  \pgfmathtruncatemacro{\fcoanti}{0}%
  \ifdim\fcoantiabs pt < 0.0001pt\relax
    \pgfmathtruncatemacro{\fcoanti}{1}%
    \pgfmathsetmacro{\fconx}{-\fcoay}%
    \pgfmathsetmacro{\fcony}{\fcoax}%
    \pgfmathsetmacro{\fconz}{0}%
  \else
    \pgfmathsetmacro{\fcocx}{(\fcoax+\fcobx)/(1+\fcodot)}%
    \pgfmathsetmacro{\fcocy}{(\fcoay+\fcoby)/(1+\fcodot)}%
    \pgfmathsetmacro{\fcor}{sqrt(max((1-\fcodot)/(1+\fcodot),0.000001))}%
    \pgfmathsetmacro{\fconx}{\fcocx/\fcor}%
    \pgfmathsetmacro{\fcony}{\fcocy/\fcor}%
    \pgfmathsetmacro{\fconz}{1/\fcor}%
  \fi
  \pgfmathsetmacro{\fctdenone}{\fcta*\fcta+\fctb*\fctb}%
  \pgfmathsetmacro{\fctdentwo}{\fctc*\fctc+\fctd*\fctd}%
  \pgfmathsetmacro{\fctax}{2*\fcta*\fctb/\fctdenone}%
  \pgfmathsetmacro{\fctay}{(\fcta*\fcta-\fctb*\fctb)/\fctdenone}%
  \pgfmathsetmacro{\fctbx}{2*\fctc*\fctd/\fctdentwo}%
  \pgfmathsetmacro{\fctby}{(\fctc*\fctc-\fctd*\fctd)/\fctdentwo}%
  \pgfmathsetmacro{\fctdot}{\fctax*\fctbx+\fctay*\fctby}%
  \pgfmathsetmacro{\fctantiabs}{abs(1+\fctdot)}%
  \pgfmathtruncatemacro{\fctanti}{0}%
  \ifdim\fctantiabs pt < 0.0001pt\relax
    \pgfmathtruncatemacro{\fctanti}{1}%
    \pgfmathsetmacro{\fctnx}{-\fctay}%
    \pgfmathsetmacro{\fctny}{\fctax}%
    \pgfmathsetmacro{\fctnz}{0}%
  \else
    \pgfmathsetmacro{\fctcx}{(\fctax+\fctbx)/(1+\fctdot)}%
    \pgfmathsetmacro{\fctcy}{(\fctay+\fctby)/(1+\fctdot)}%
    \pgfmathsetmacro{\fctr}{sqrt(max((1-\fctdot)/(1+\fctdot),0.000001))}%
    \pgfmathsetmacro{\fctnx}{\fctcx/\fctr}%
    \pgfmathsetmacro{\fctny}{\fctcy/\fctr}%
    \pgfmathsetmacro{\fctnz}{1/\fctr}%
  \fi
  \pgfmathsetmacro{\fcbraw}{\fconx*\fctnx+\fcony*\fctny-\fconz*\fctnz}%
  \pgfmathsetmacro{\fcbabs}{abs(\fcbraw)}%
  \pgfmathtruncatemacro{\fcbad}{0}%
  \ifdim\fcbabs pt < 1.0001pt\relax
    \pgfmathtruncatemacro{\fcbad}{1}%
    \PackageWarning{fareysegment}{The prescribed Farey lines intersect or are asymptotic, so there is no finite common perpendicular}%
  \fi
  \ifnum\fcbad=0
    \ifdim\fcbraw pt < 0pt\relax
      \pgfmathsetmacro{\fctnx}{-\fctnx}%
      \pgfmathsetmacro{\fctny}{-\fctny}%
      \pgfmathsetmacro{\fctnz}{-\fctnz}%
      \pgfmathsetmacro{\fcb}{-\fcbraw}%
    \else
      \pgfmathsetmacro{\fcb}{\fcbraw}%
    \fi
    \pgfmathsetmacro{\fcsqrt}{sqrt(max(\fcb*\fcb-1,0.000001))}%
    \pgfmathsetmacro{\fcHXone}{(\fctnx-\fcb*\fconx)/\fcsqrt}%
    \pgfmathsetmacro{\fcHYone}{(\fctny-\fcb*\fcony)/\fcsqrt}%
    \pgfmathsetmacro{\fcHTone}{(\fctnz-\fcb*\fconz)/\fcsqrt}%
    \pgfmathsetmacro{\fcHXtwo}{(\fcb*\fctnx-\fconx)/\fcsqrt}%
    \pgfmathsetmacro{\fcHYtwo}{(\fcb*\fctny-\fcony)/\fcsqrt}%
    \pgfmathsetmacro{\fcHTtwo}{(\fcb*\fctnz-\fconz)/\fcsqrt}%
    \ifdim\fcHTone pt < 0pt\relax
      \pgfmathsetmacro{\fcHXone}{-\fcHXone}%
      \pgfmathsetmacro{\fcHYone}{-\fcHYone}%
      \pgfmathsetmacro{\fcHTone}{-\fcHTone}%
      \pgfmathsetmacro{\fcHXtwo}{-\fcHXtwo}%
      \pgfmathsetmacro{\fcHYtwo}{-\fcHYtwo}%
      \pgfmathsetmacro{\fcHTtwo}{-\fcHTtwo}%
    \fi
    \pgfmathsetmacro{\fcfx}{\fcHXone/(\fcHTone+1)}%
    \pgfmathsetmacro{\fcfy}{\fcHYone/(\fcHTone+1)}%
    \pgfmathsetmacro{\fcgx}{\fcHXtwo/(\fcHTtwo+1)}%
    \pgfmathsetmacro{\fcgy}{\fcHYtwo/(\fcHTtwo+1)}%
    \pgfmathsetmacro{\fcmx}{\fconz*\fctny-\fcony*\fctnz}%
    \pgfmathsetmacro{\fcmy}{\fconx*\fctnz-\fconz*\fctnx}%
    \pgfmathsetmacro{\fcmz}{\fconx*\fctny-\fcony*\fctnx}%
    \pgfmathsetmacro{\fcmzabs}{abs(\fcmz)}%
    \ifdim\fcmzabs pt < \fcDiameterTolerance pt\relax
      \draw[#1] (\fcfx,\fcfy)--(\fcgx,\fcgy);%
      \pgfmathsetmacro{\fclx}{(\fcfx+\fcgx)/2}%
      \pgfmathsetmacro{\fcly}{(\fcfy+\fcgy)/2}%
      \node[fc segment label] at (\fclx,\fcly) {#4};%
      \pgfmathsetmacro{\fcuax}{\fcgx-\fcfx}%
      \pgfmathsetmacro{\fcuay}{\fcgy-\fcfy}%
      \pgfmathsetmacro{\fcubx}{\fcfx-\fcgx}%
      \pgfmathsetmacro{\fcuby}{\fcfy-\fcgy}%
    \else
      \pgfmathsetmacro{\fccccx}{\fcmx/\fcmz}%
      \pgfmathsetmacro{\fcccy}{\fcmy/\fcmz}%
      \pgfmathsetmacro{\fccr}{sqrt(max(\fccccx*\fccccx+\fcccy*\fcccy-1,0.000001))}%
      \pgfmathsetmacro{\fcanga}{atan2(\fcfy-\fcccy,\fcfx-\fccccx)}%
      \pgfmathsetmacro{\fcangb}{atan2(\fcgy-\fcccy,\fcgx-\fccccx)}%
      \pgfmathsetmacro{\fcdiff}{\fcangb-\fcanga}%
      \ifdim\fcdiff pt > 180pt\relax
        \pgfmathsetmacro{\fcangb}{\fcangb-360}%
      \fi
      \ifdim\fcdiff pt < -180pt\relax
        \pgfmathsetmacro{\fcangb}{\fcangb+360}%
      \fi
      \draw[#1] (\fcfx,\fcfy) arc[start angle=\fcanga,end angle=\fcangb,radius=\fccr];%
      \pgfmathsetmacro{\fcmidang}{(\fcanga+\fcangb)/2}%
      \pgfmathsetmacro{\fclx}{\fccccx+\fccr*cos(\fcmidang)}%
      \pgfmathsetmacro{\fcly}{\fcccy+\fccr*sin(\fcmidang)}%
      \node[fc segment label,left] at (\fclx,\fcly) {#4};%
      \pgfmathsetmacro{\fcdiff}{\fcangb-\fcanga}%
      \pgfmathsetmacro{\fcuax}{-(\fcfy-\fcccy)}%
      \pgfmathsetmacro{\fcuay}{\fcfx-\fccccx}%
      \pgfmathsetmacro{\fcubx}{\fcgy-\fcccy}%
      \pgfmathsetmacro{\fcuby}{-(\fcgx-\fccccx)}%
      \ifdim\fcdiff pt < 0pt\relax
        \pgfmathsetmacro{\fcuax}{-\fcuax}%
        \pgfmathsetmacro{\fcuay}{-\fcuay}%
        \pgfmathsetmacro{\fcubx}{-\fcubx}%
        \pgfmathsetmacro{\fcuby}{-\fcuby}%
      \fi
    \fi
    \ifnum\fcoanti=1
      \pgfmathsetmacro{\fcvax}{\fcobx-\fcoax}%
      \pgfmathsetmacro{\fcvay}{\fcoby-\fcoay}%
    \else
      \pgfmathsetmacro{\fcvax}{-(\fcfy-\fcocy)}%
      \pgfmathsetmacro{\fcvay}{\fcfx-\fcocx}%
      \pgfmathsetmacro{\fcvanga}{atan2(\fcoay-\fcocy,\fcoax-\fcocx)}%
      \pgfmathsetmacro{\fcvangb}{atan2(\fcoby-\fcocy,\fcobx-\fcocx)}%
      \pgfmathsetmacro{\fcvdiff}{\fcvangb-\fcvanga}%
      \ifdim\fcvdiff pt > 180pt\relax
        \pgfmathsetmacro{\fcvdiff}{\fcvdiff-360}%
      \fi
      \ifdim\fcvdiff pt < -180pt\relax
        \pgfmathsetmacro{\fcvdiff}{\fcvdiff+360}%
      \fi
      \ifdim\fcvdiff pt < 0pt\relax
        \pgfmathsetmacro{\fcvax}{-\fcvax}%
        \pgfmathsetmacro{\fcvay}{-\fcvay}%
      \fi
    \fi
    \fcRightAngleMark{\fcfx}{\fcfy}{\fcuax}{\fcuay}{\fcvax}{\fcvay}{\fcRightAngleSize}%
    \ifnum\fctanti=1
      \pgfmathsetmacro{\fcvbx}{\fctbx-\fctax}%
      \pgfmathsetmacro{\fcvby}{\fctby-\fctay}%
    \else
      \pgfmathsetmacro{\fcvbx}{-(\fcgy-\fctcy)}%
      \pgfmathsetmacro{\fcvby}{\fcgx-\fctcx}%
      \pgfmathsetmacro{\fcvanga}{atan2(\fctay-\fctcy,\fctax-\fctcx)}%
      \pgfmathsetmacro{\fcvangb}{atan2(\fctby-\fctcy,\fctbx-\fctcx)}%
      \pgfmathsetmacro{\fcvdiff}{\fcvangb-\fcvanga}%
      \ifdim\fcvdiff pt > 180pt\relax
        \pgfmathsetmacro{\fcvdiff}{\fcvdiff-360}%
      \fi
      \ifdim\fcvdiff pt < -180pt\relax
        \pgfmathsetmacro{\fcvdiff}{\fcvdiff+360}%
      \fi
      \ifdim\fcvdiff pt < 0pt\relax
        \pgfmathsetmacro{\fcvbx}{-\fcvbx}%
        \pgfmathsetmacro{\fcvby}{-\fcvby}%
      \fi
    \fi
    \fcRightAngleMark{\fcgx}{\fcgy}{\fcubx}{\fcuby}{\fcvbx}{\fcvby}{\fcRightAngleSize}%
  \fi
  \endgroup
}

\AtEndDocument{%
	\par
	\medskip
	\begin{tabular}{@{}l@{}}%
		{School of Mathematics and Statistics, The Open University,}\\
		{Milton Keynes, MK7 6AA, United Kingdom}\\
		\textit{E-mail address}: \texttt{ian.short@open.ac.uk}\\
		\textit{E-mail address}: \texttt{andrei.zabolotskii@open.ac.uk}
\end{tabular}}

\title{ All Y-friezes come from \SL2-friezes}
\author{Ian Short and Andrei Zabolotskii}
\date{}
\begin{document}

\maketitle

\begin{abstract}
We resolve a conjecture of de Saint Germain: all Y-friezes arise from \SL2-friezes. Additionally, we give a threefold characterisation of those \SL2-friezes that give rise to the same Y-frieze: an elementary characterisation using quiddities, a combinatorial characterisation using triangulated polygons, and a geometrical characterisation using chains of horocycles.
\end{abstract}

\section{Introduction}
\label{sec:intro}

An \SL2-frieze is a staggered array of integers bounded by $1$'s in which every diamond of entries
\begin{tikzpicture}[baseline=-.5ex,scale=0.3]
  \draw[gray!30] (0,1) -- (1,0) -- (0,-1) -- (-1,0) -- cycle;
  \node at (0,1) {\scriptsize $N$};
  \node at (0,-1) {\scriptsize $S$};
  \node at (-1,0) {\scriptsize $W$};
  \node at (1,0) {\scriptsize $E$};
\end{tikzpicture}
satisfies the \emph{SL$_{\text{\textit 2}}$-diamond rule}
\[
EW - NS = 1.
\]
A Y-frieze is a staggered array bounded by $0$'s in which every diamond satisfies the \emph{Y-diamond rule}
\[
EW = (N+1)(S+1).
\]

In the 1970s, Coxeter \cite{Co1971} and Conway introduced frieze patterns -- which we refer to as \SL2-friezes to contrast them to their Y counterparts -- and established a correspondence between \SL2-friezes and triangulated polygons.
In \cite{MoOvTa2015}, Morier-Genoud, Ovsienko, and Tabachnikov used the Farey graph $\mathscr{F}$ to interpret that correspondence geometrically. We paraphrase their result as follows.

\begin{maintheorem}{A}\label{theoremA}
There is a one-to-one correspondence
\[
\textnormal{SL}_2(\mathbb{Z})\Big\backslash\mleft\{\parbox{4.6cm}{\centering\textnormal{simple closed clockwise paths in $\mathscr{F}$ of length $n+3$}}\mright\}\quad \longrightarrow\quad \mleft\{\parbox{1.9cm}{\centering\textnormal{$\text{SL}_2$-friezes of width $n$}}\mright\}
\]
determined by 
\[
A_{i,j} = a_jb_i-b_ja_i,\qquad\text{for $1\leq i-j\leq n+2$},
\]
where $a_i/b_i$ is a simple closed clockwise path in $\mathscr{F}$ satisfying $a_{i-1}b_i-b_{i-1}a_i=1$, for $i\in\mathbb{Z}$.
\end{maintheorem}

Here $A_{i,j} $ is the $(i,j)$th entry of the $\text{SL}_2$-frieze $A$ and $(a_0/b_0, a_1/b_1,\dots,a_{n+3}/b_{n+3})$, where $(a_{n+3},b_{n+3})=-(a_0,b_0)$, is a clockwise sequence of reduced rationals that completes exactly one simple cycle, with the clockwise order on the extended real line inherited from the unit circle by stereographic projection. The condition $a_{i-1}b_i-b_{i-1}a_i=1$ merely ensures that we choose the sign of the numerator and denominator of $a_i/b_i$ judiciously (since $a/b$ and $(-a)/(-b)$ represent the same rational). The path can be extended to a biinfinite periodic path by defining $(a_{i+n+3},b_{i+n+3})=-(a_i,b_i)$, for $i\in\mathbb{Z}$. It encloses a triangulated polygon in $\mathscr{F}$; for example, the path of length 7 shown in Figure~\ref{figure1} encloses a triangulated heptagon.

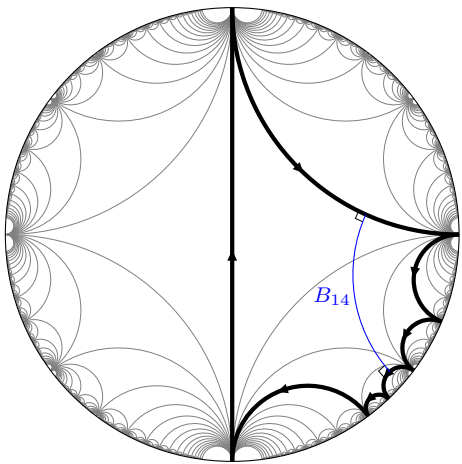
\begin{figure}[ht!]
\centering
\begin{tikzpicture}[scale=3]

\discfarey

\def\vertices{{{1,0},{1,1},{2,3},{1,2},{2,5},{1,3},{0,1},{1,0}}}

\foreach \i in {0,...,6} {
    \pgfmathsetmacro{\numA}{\vertices[\i][0]}
    \pgfmathsetmacro{\denA}{\vertices[\i][1]}

    \pgfmathtruncatemacro{\nexti}{\i+1}
    \pgfmathsetmacro{\numB}{\vertices[\nexti][0]}
    \pgfmathsetmacro{\denB}{\vertices[\nexti][1]}

     \shline[ultra thick](\numA:\denA:\numB:\denB); 
}

\fareysegment[blue]{1:1:1:0}{1:2:2:5}{\footnotesize\color{blue}{$B_{14}$}}

\end{tikzpicture}
\caption{Simple closed Farey path}
\label{figure1}
\end{figure}

The \emph{quiddity} of the path is the sequence $\sigma_i=a_{i-1}b_{i+1}-b_{i-1}a_{i+1}$, for $i\in\mathbb{Z}$. The positive integer $\sigma_i$ is equal to the number of inner triangles in the triangulated polygon incident to the $i$th vertex in the cycle. The quiddity is invariant under the action of $\text{SL}_2(\mathbb{Z})$ on pairs $\begin{psmallmatrix}a_i\\b_i\end{psmallmatrix}$ by left multiplication. Likewise, the \SL2-frieze $A_{i,j}$ is unchanged by replacing $a_i/b_i$ with an $\SL2(\Z)$-equivalent path.

The entries in an \SL2-frieze have several other interpretations; in particular, they are ``shadows'' of cluster variables in cluster algebras of type $A_n$. Y-frieze patterns were introduced by de Saint Germain in \cite{dSG2023} as ``shadows'' of a subset of Y-variables from the same cluster algebras; we refer to them as Y-friezes for brevity, even though this term has a different meaning in \cite{dSG2023}.

The following theorem relates Y-friezes to \SL2-friezes.

\begin{maintheorem}{B}\label{theorem:surj}
There is a surjective map
\[
p_n\colon\,
\mleft\{\parbox{1.9cm}{\centering\textnormal{$\text{SL}_2$-friezes of width $n$}}\mright\}
\quad \longrightarrow\quad
\mleft\{\parbox{1.9cm}{\centering\textnormal{Y-friezes of~width $n$}}\mright\}
\]
determined by
\[
p_n(A)_{i,j} = A_{i,j+1}A_{i+1,j}.
\]

At most two \SL2-friezes lie in the same fibre of $p_n$.
Specifically, two \SL2-friezes with distinct quiddities $\sigma_i$ and $\sigma_i'$ map to the same Y-frieze if and only if $n$ is odd and $\sigma'_o=\lambda\sigma_o$ and $\sigma_e=\lambda\sigma'_e$, for all odd and even indices $o$ and $e$, where $\lambda$ is $\tfrac13$, $\tfrac12$, 2, or 3.
\end{maintheorem}

That $p_n$ is well-defined follows (essentially) from \cite[Theorem~A]{dSG2023}.
It is also noted in \cite{dSG2023} that $p_n$ is not injective when $n$ is odd. The most original part of Theorem~\ref{theorem:surj} is the surjectivity assertion. This assertion resolves a conjecture of  de Saint Germain first stated in \cite[Section 1.1]{dSG2023} and then restated in \cite[Conjecture 10]{dSG2025}.

An independent and different proof of Theorem~\ref{theorem:surj} has been obtained by Hin Chung Henry Tsang and Jonathan Wilson (in preparation).

By combining Theorems \ref{theoremA} and \ref{theorem:surj}, we can deduce a similar theorem to Theorem~\ref{theoremA} involving the Farey graph $\mathscr{F}$ for Y-friezes.

\begin{maintheorem}{C}\label{theoremC}
There is a surjective map
\[
\textnormal{SL}_2(\mathbb{Z})\Big\backslash\mleft\{\parbox{4.6cm}{\centering\textnormal{simple closed clockwise paths in $\mathscr{F}$ of length $n+3$}}\mright\}\quad \longrightarrow\quad \mleft\{\parbox{1.9cm}{\centering\textnormal{Y-friezes of~width $n$}}\mright\}
\]
determined by 
\[
B_{i,j} = \left[\tfrac{a_{i}}{b_{i}},\tfrac{a_{i+1}}{b_{i+1}},\tfrac{a_{j}}{b_{j}},\tfrac{a_{j+1}}{b_{j+1}}\right],\qquad\text{for $1\leq i-j\leq n+2$},
\]
where $a_i/b_i$ is a simple closed clockwise path in $\mathscr{F}$  satisfying $a_{i-1}b_i-b_{i-1}a_i=1$, for $i\in\mathbb{Z}$.

\end{maintheorem}
Here $B_{i,j}$ is the $(i,j)$th entry of the Y-frieze $B$ and, as before, $a_i/b_i$ are the vertices of a simple closed clockwise path in $\mathscr{F}$ (which we also think of as a biinfinite periodic path), and 
\[
[p,q,r,s] = \frac{p-s}{p-q}\frac{q-r}{r-s}.
\]
is the cross-ratio of four points $p,q,r,s$ on the extended real line.
Notice that $B_{i,j} = (a_ib_{j+1}-a_{j+1}b_i)(a_{i+1}b_j-a_jb_{i+1})$; that is,
the map defined in Theorem \ref{theoremC} is the composition of the maps from Theorems \ref{theoremA} and \ref{theorem:surj}, so it is surjective.

We introduce Y-friezes and \SL2-friezes formally in Section~\ref{sec:defs}. In Section~\ref{sec:proof}, we prove the surjectivity claim of Theorem~\ref{theorem:surj}. In Section~\ref{sec:noninj}, we give a combinatorial criterion for non-injectivity of $p_n$ and prove the rest of Theorem~\ref{theorem:surj}. In Section~\ref{sec:hyper} we explore the geometry behind Y-friezes.

\section{$\textbf{SL}_\textbf{2}$-friezes and Y-friezes}
\label{sec:defs}

We assume throughout this section and beyond that $n$ is a nonnegative integer.

\begin{definition}
Let $a\in\Z$.
A \emph{staggered array} of width $n$ bounded by $a$'s is a map $m\colon \{(i,j)\in\mathbb{Z}\times\mathbb{Z} : 1\leq i-j\leq n+2\}\longrightarrow \mathbb{Z}$ with
\begin{itemize}
\item $m_{i+1,i}=m_{i+n+2,i}=a$, for $i\in\mathbb{Z}$, and
\item $m_{i,j}>0$, for $2\leq i-j\leq n+1$.
\end{itemize}
\end{definition}

A staggered array is usually visualised as $n+2$ staggered infinite rows of integers. For example, a~staggered array bounded by 0's with $n$ odd looks like this:
\[
\hspace*{-45pt}\vcenter{\footnotesize
	\xymatrix @-0.3pc @!0 {
0&&0&&0&&0&&0&&0&&0&&0&&0&&0\\
	& m_{-3,-5} && m_{-2,-4} && m_{-1,-3} && m_{0,-2} && m_{1,-1} && m_{2,0} && m_{3,1} && m_{4,2}&& m_{5,3} & \\
			m_{-3,-6}	&& m_{-2,-5} 	  && m_{-1,-4}	&& m_{0,-3} && m_{1,-2} && m_{2,-1} && m_{3,0} && m_{4,1}  && m_{5,2} && m_{6,3}  \\
		& m_{-2,-6} && m_{-1,-5} && m_{0,-4} && m_{1,-3} && m_{2,-2} && m_{3,-1} && m_{4,0} && m_{5,1}&& m_{6,2} & \\
			m_{-2,-7}	&& m_{-1,-6} 	  && m_{0,-5}	&& m_{1,-4} && m_{2,-3} && m_{3,-2} && m_{4,-1} && m_{5,0}  && m_{6,1} && m_{7,2}  \\
		& m_{-1,-7} && m_{0,-6} && m_{1,-5} && m_{2,-4} && m_{3,-3} && m_{4,-2} && m_{5,-1} && m_{6,0}&& m_{7,1} & \\		
&  &&&&&  && \!\!\!\!\!\!\!\!\ddots &&  && && && \\ 
&m_{n-7,-8}	&& m_{n-6,-7} 	  && m_{n-5,-6}	&& m_{n-4,-5} && m_{n-3,-4} && m_{n-2,-3} && m_{n-1,-2} && m_{n,-1}  && m_{n+1,0} && \\
&&0&&0&&0&&0&&0&&0&&0&&0&&0
	}
}
\]

\begin{definition}
An \emph{\SL2-frieze} $A$ of width $n$ is a staggered array of width $n$ bounded by $1$'s that satisfies the \SL2-diamond rule $A_{i,j}A_{i+1,j+1}-A_{i,j+1}A_{i+1,j} = 1$, for $2\leq i-j\leq n+1$.
\end{definition}

The row below the top row of an \SL2-frieze is called the \emph{quiddity}.
Through Theorem~\ref{theoremA} it agrees with the quiddity of a corresponding path in $\mathscr{F}$.

We will also work with real \SL2-friezes and rational \SL2-friezes, which are defined in the same way as (integral) \SL2-friezes but with positive real or rational entries instead of integers.

\begin{definition}
A \emph{Y-frieze} $B$ of width $n$ is a staggered array of width $n$ bounded by $0$'s that satisfies the Y-diamond rule $(B_{i,j+1}+1)(B_{i+1,j}+1)=B_{i,j}B_{i+1,j+1}$, for $2\leq i-j\leq n+1$.
\end{definition}

\begin{remark}
In our layout for \SL2-friezes and Y-friezes, both coordinate axes associated with the indices $i,j$ are directed along diagonals (as in \cite{Mo2015}), not along rows (as in \cite{dSG2025}). Our notion of ``width'', the number of nontrivial rows $n$, agrees with e.g.\ \cite{dSG2023} but differs from e.g.\ \cite{ShVaZa2025}, where ``width'' is $N=n+3$.
\end{remark}

Each row of an \SL2-frieze or Y-frieze is periodic with (not necessarily minimum) period $N=n+3$; moreover, both \SL2-friezes \cite{Co1971} and Y-friezes \cite[Theorem~B]{dSG2023} possess a glide reflection symmetry $m_{i,j}=m_{j+N,i}$, for $1\leqslant i-j\leqslant n+2$.

We recall the map $p_n$ from Theorem \ref{theorem:surj} (also \cite[Theorem 8]{dSG2025}), which sends \SL2-friezes to Y-friezes according to the rule
\[
B_{i,j} = A_{i,j+1}A_{i+1,j},
\]
where $B=p_n(A)$.
In other words, each element of the staggered array $B$ is obtained as the product of two vertically adjacent elements in $A$. The array $B$ is indeed a Y-frieze, because the Y-diamond rule for $B$ follows from the \SL2-diamond rule for $A$:
\[
\begin{aligned}
&(B_{i,j+1}+1)(B_{i+1,j}+1) \\&= (A_{i,j+2}A_{i+1,j+1}+1)\times(A_{i+1,j+1}A_{i+2,j}+1)  \\& = A_{i,j+1}A_{i+1,j+2}\times A_{i+1,j}A_{i+2,j+1} \\&= A_{i,j+1}A_{i+1,j} \times A_{i+1,j+2}A_{i+2,j+1}   \\ &= B_{i,j}B_{i+1,j+1}.
\end{aligned}
\]
The second row of the Y-frieze $B$ is 
\[
B_{i,i-2}=(a_{i-1}b_{i}-b_{i-1}a_{i})(a_{i-2}b_{i+1}-b_{i-2}a_{i+1})=a_{i-2}b_{i+1}-b_{i-2}a_{i+1}=A_{i+1,i-2}.
\]
This is the third row of the \SL2-frieze $A$, directly below the quiddity.

In the following examples we display entries of \SL2-friezes in black and entries of Y-friezes in grey.
\begin{example}
\label{ex:basic}

This is an example of an \SL2-frieze $A$ (left) and a Y-frieze $B=p_n(A)$ (right) with $n=3$. All rows are 6-periodic.

\(
\small
\xymatrix @-1pc @!0{
&&1&&1&&1&&1&&1&&1&&1 \\
&1&&2&&2&&2&&1&&4&&1 \\
\cdots&&1&&3&&3&&1&&3&&3&&1&\cdots  \\
&2&&1&&4&&1&&2&&2&&2 \\
&&1&&1&&1&&1&&1&&1&&1
}
\qquad
{
\color{gray}
\xymatrix @-1pc @!0{
&0&&0&&0&&0&&0&&0&&0 \\
&&1&&3&&3&&1&&3&&3&&1 \\
\cdots&2&&2&&8&&2&&2&&8&&2&\cdots  \\
&&1&&3&&3&&1&&3&&3&&1 \\
&0&&0&&0&&0&&0&&0&&0
}
}
\)

We can superimpose $A$ and $B$ by placing between each pair of vertically adjacent \SL2-frieze entries their product, which is the entry of the Y-frieze with the same indices $i,j$ as the \SL2-frieze entry to its left.

\(
\small
\xymatrix @-1pc @!0{
&{\color{gray}0}&1&{\color{gray}0}&1&{\color{gray}0}&1&{\color{gray}0}&1&{\color{gray}0}&1&{\color{gray}0}&1&{\color{gray}0}&1 \\
&1&{\color{gray}1}&2&{\color{gray}3}&2&{\color{gray}3}&2&{\color{gray}1}&1&{\color{gray}3}&4&{\color{gray}3}&1&{\color{gray}1}  \\
\cdots&{\color{gray}2}&1&{\color{gray}2}&3&{\color{gray}8}&3&{\color{gray}2}&1&{\color{gray}2}&3&{\color{gray}8}&3&{\color{gray}2}&1&\cdots  \\
&2&{\color{gray}1}&1&{\color{gray}3}&4&{\color{gray}3}&1&{\color{gray}1}&2&{\color{gray}3}&2&{\color{gray}3}&2&{\color{gray}1} \\
&{\color{gray}0}&1&{\color{gray}0}&1&{\color{gray}0}&1&{\color{gray}0}&1&{\color{gray}0}&1&{\color{gray}0}&1&{\color{gray}0}&1
}
\)
\end{example}

\begin{example}
\label{ex:other}
We can adjust the \SL2-frieze $A$ from Example \ref{ex:basic} by multiplying the entries of every \emph{even} row alternately by $2$ and $\tfrac12$ (shaded \frmul{red} and \frdiv{grey} respectively). The resulting array is a different \SL2-frieze $A'$, which has the same image under $p_n$ as $A$; that is $p_n(A') = p_n(A) = B$.

\(
\small
\xymatrix @-1pc @!0{
&{\color{gray}0}&1&{\color{gray}0}&1&{\color{gray}0}&1&{\color{gray}0}&1&{\color{gray}0}&1&{\color{gray}0}&1&{\color{gray}0}&1 \\
&\frmul2&{\color{gray}1}&\frdiv1&{\color{gray}3}&\frmul4&{\color{gray}3}&\frdiv1&{\color{gray}1}&\frmul2&{\color{gray}3}&\frdiv2&{\color{gray}3}&\frmul2&{\color{gray}1}  \\
\cdots&{\color{gray}2}&1&{\color{gray}2}&3&{\color{gray}8}&3&{\color{gray}2}&1&{\color{gray}2}&3&{\color{gray}8}&3&{\color{gray}2}&1&\cdots  \\
&\frdiv1&{\color{gray}1}&\frmul2&{\color{gray}3}&\frdiv2&{\color{gray}3}&\frmul2&{\color{gray}1}&\frdiv1&{\color{gray}3}&\frmul4&{\color{gray}3}&\frdiv1&{\color{gray}1} \\
&{\color{gray}0}&1&{\color{gray}0}&1&{\color{gray}0}&1&{\color{gray}0}&1&{\color{gray}0}&1&{\color{gray}0}&1&{\color{gray}0}&1
}
\)
\end{example}

\section{Proof of the surjectivity conjecture}
\label{sec:proof}

We assume throughout the following sections that $B$ is a Y-frieze of width $n$.

\subsection{Reconstruction by columns yields integers}

Let $k\in\mathbb{Z}$. Let $x_{i,k}=B_{k+i,k-i}$ and $y_{i,k}=B_{k+i+1,k-i}$, for $1\leq i\leq \lceil n/2\rceil$. For fixed $k$, the entries $x_{i,k}$ occupy one of the vertical columns from $B$ without a 0 at the top, and $y_{i,k}$ occupy the vertical column with a 0  at the top to the right of the column of $x_{i,k}$. Let $X_{-1,k}=-1$, $X_{0,k}=1$ and $X_{i,k}=x_{i,k}/X_{i-1,k}$, for $1\leq i\leq\lceil n/2\rceil$.

\begin{lemma}
\label{lemma:x1XX}

We have $y_{i,k}+1=X_{i,k}X_{i,k+1}$, for $0\leq i\leq \lceil n/2\rceil$.
\end{lemma}
\begin{proof}
We proceed by induction on $i$. We have $y_{0,k}+1=X_{0,k}X_{0,k+1}=1$, for $k\in\mathbb{Z}$. Suppose now that $y_{i,k}+1=X_{i,k}X_{i,k+1}$. Then, by the Y-diamond rule,
\[
y_{i+1,k}+1= \frac{x_{i+1,k}x_{i+1,k+1}}{y_{i,k}+1}=\frac{x_{i+1,k}x_{i+1,k+1}}{X_{i,k}X_{i,k+1}}=X_{i+1,k}X_{i+1,k+1},
\]
as required.
\end{proof}

\begin{lemma}\label{lemma2}
The numbers $X_{i,k}$, for $0\leq i\leq  \lceil n/2\rceil$, are positive integers.
\end{lemma}
\begin{proof}
We proceed by induction on $i$. The assertion is true for $i=0$. Suppose then that $X_{i,k}$ is a positive integer for some $i\geq 0$ and all $k\in\mathbb{Z}$. By the Y-diamond rule and Lemma~\ref{lemma:x1XX},
\[
(x_{i+1,k}+1)(x_{i,k}+1)=y_{i,k}y_{i,k-1}=(X_{i,k}X_{i,k+1}-1)(X_{i,k-1}X_{i,k}-1).
\]
Now, $x_{i,k}=X_{i,k}X_{i-1,k}$, so 
\[
x_{i+1,k}=X_{i,k}(X_{i,k+1}X_{i,k-1}X_{i,k}-X_{i,k-1}-X_{i,k+1}-x_{i+1,k}X_{i-1,k}-X_{i-1,k}).
\]
Consequently, $X_{i+1,k}=x_{i+1,k}/X_{i,k}$ is a positive integer, as required.
\end{proof}

Notice that $X_{i,k}$ is periodic in $k$ with period $N=n+3$.

\subsection{Injective case: $n$ even}
Suppose $n$ is even. Let $N=n+3$ (a period of $B$), so $N$ is odd. Here we construct a unique \SL2-frieze $A$ such that $p_n(A) = B$.

For odd rows of $A$, we define $A_{k+i+1,k-i}=X_{i,k}$ for $0\leq i\leq n/2$ and $k\in\mathbb{Z}$. The even rows are then determined by the glide reflection symmetry $A_{i,j}=A_{j+N,i}$.

The top and bottom rows of $A$ are all $1$'s and, by Lemma~\ref{lemma2}, all elements of $A$ are positive integers. For the diamond rule, using Lemma~\ref{lemma:x1XX}, we have
\[
\begin{aligned}
&A_{k+i+1,k-i}A_{k+i+2,k+1-i} - A_{k+i+2,k-i}A_{k+i+1,k+1-i} \\ 
&= A_{k+i+1,k-i}A_{k+i+2,k+1-i} - A_{N+k-i,k+i+2}A_{N+k+1-i,k+i+1} \\
&= X_{i,k}X_{i,k+1} - X_{n/2-i,n/2+2+k}X_{n/2+1-i,n/2+2+k}  \\
&=(y_{i,k}+1) - x_{n/2+1-i,n/2+2+k} \\
&= B_{k+i+1,k-i} + 1 - B_{N+k-i,k+i+1} = 1.
\end{aligned}
\]
With a similar calculation we can verify the diamond rule for diamonds centred on even rows. Therefore $A$ is an \SL2-frieze. Next, for even rows we have
\[
A_{i,j+1}A_{i+1,j} = X_{(i-j)/2-1,(i+j)/2}X_{(i-j)/2,(i+j)/2}= x_{(i-j)/2,(i+j)/2} = B_{i,j},\]
and a similar calculation can be performed for odd rows. Hence $p_n(A) = B$, so $p_n$ is surjective.

Finally, we observe that the equation $p_n(A)=B$ implies that $A_{k+i+1,k-1}=X_{i,k}$ (by an induction argument). Consequently, $A$ is uniquely specified by this equation and the glide reflection symmetry, so $p_n$ is injective when $n$ is even.

\subsection{Non-injective case: $n$ odd}

Now suppose that the width $n$ is odd, so $N$ is even.

\begin{definition}
\label{def:rescaling}
Let $A$ be a real \SL2-frieze, and let $\lambda>0$. Then
\[
A'_{i,j} =
\left\{
\begin{array}{rl}
A_{i,j} & \text{if $i-j$ odd,}\\
\lambda^{(-1)^{i+1}}A_{i,j} & \text{if $i-j$ even}
\end{array}
\right.
\]
is also a real \SL2-frieze. We say that $A'$ is the result of \emph{rescaling} $A$ by a factor of $\lambda$.
\end{definition}
This action of positive reals on real \SL2-friezes with an odd number of nontrivial rows is precisely the action from \cite[Theorem 1.2]{ShVaZa2025}.

The quiddities $\sigma_i$ and $\sigma'_i$ of $A$ and $A'$ are related by
\[
\sigma'_i =
\left\{
\begin{array}{rl}
\lambda \sigma_i & \text{if $i$ even,}\\
\lambda^{-1} \sigma_i & \text{if $i$ odd.}
\end{array}
\right.
\]

The \SL2-friezes in Examples \ref{ex:basic} and \ref{ex:other} are related by rescaling by a factor of~2. Any two \SL2-friezes that are related by a rescaling are mapped to the same Y-frieze by $p_n$.

To prove the surjectivity conjecture for odd $n$, we first construct a \emph{rational} \SL2-frieze $A$ satisfying $A_{i,j+1}A_{i+1,j}=B_{i,j}$. We begin by defining the quiddity of $A$: positive rational numbers $a_{1,k}$, for $k\in\mathbb{Z}$, with $a_{1,k-1}a_{1,k}=x_{1,k}+1$ (starting from arbitrary $a_{1,0}$). Then define $a_{r+1,k}=y_{r,k}/a_{r,k}$, for $1\leq r \leq (n-1)/2$. By induction on $r$, we have
\[
a_{r,k-1}a_{r,k}=x_{r,k}+1.
\]
We define $A_{k+i+1,k-i}=X_{i,k}$ and $A_{k+i+1,k-i+1}=a_{i,k}$, for $k\in\mathbb{Z}$. This is a rational $\text{SL}_2$-frieze because
\[
X_{i,k}X_{i,k+1}-a_{i,k}a_{i+1,k}=(y_{i,k}+1)-y_{i,k}=1
\]
and
\[
a_{i,k-1}a_{i,k}-X_{i-1,k}X_{i,k}=(x_{i,k}+1)-x_{i,k}=1.
\]
The rational $\text{SL}_2$-frieze $A$ is periodic with period $N$. As such, it satisfies the Ptolemy relation \cite{Mo2015}
\[
A_{\beta,\alpha}A_{\delta,\gamma}=A_{\gamma,\alpha}A_{\delta,\beta}-A_{\gamma,\beta}A_{\delta,\alpha},
\]
where $\alpha<\beta<\gamma<\delta<\alpha+N$. One can check that $A_{i,j+1}A_{i+1,j}=B_{i,j}$.

\begin{lemma}
The rationals $a_{1,k}$, for $k\in\mathbb{Z}$, can be chosen to be integers.
\end{lemma}
\begin{proof}
Let $b_k=a_{1,k}$; then $b_{k+N}=b_k$. We first claim that $b_ub_v\in\mathbb{N}$ whenever $d=u-v$ is odd. This is true for $d=N-1$ and $d=1$ because $b_kb_{k+N-1}=b_{k}b_{k-1}=x_{1,k}+1$.  Now take $d\in\{3,5,\dots,n\}$. From the Ptolemy relation we get 
\[
b_kb_{k+d}=A_{k+2,k}A_{k+d+2,k+d}=A_{k+d,k}A_{k+d+2,k+2}-A_{k+d,k+2}A_{k+d+2,k}.
\]
The entries on the right are all integers, by Lemma~\ref{lemma2}, since $d$ is odd. The full claim for $d$ odd follows from periodicity.

This implies that for any $e$ even and $o$ odd, the denominator of $b_e$ written as a reduced fraction divides the numerator of $b_o$ written as a reduced fraction, and therefore the least common denominator $\varepsilon$ of all $b_e$'s divides the numerator of $b_o$; and similarly, the least common denominator $\omega$ of all $b_o$'s divides the numerator of any $b_e$.

Now rescale the rational \SL2-frieze $A$ by a factor of $\varepsilon/\omega$. The resulting quiddity $(a'_{1,k})$ comprises positive integers and defines an \SL2-frieze $A'$ that still satisfies $A'_{i,j+1}A'_{i+1,j}=B_{i,j}$, as required.
\end{proof}

We can assume, then, that the quiddity of $A$ consists of positive integers, in which case all entries of $A$ must be positive integers.
Since $p_n(A)=B$ we deduce that $p_n$ is surjective.

\section{Dissected polygons and non-injectivity}
\label{sec:noninj}
To complete the proof of Theorem~\ref{theorem:surj}, it remains to characterise the non-injectivity of $p_n$ in the case of $n$ odd.

\begin{lemma}
\label{lemma:23}
Any (integral) \SL2-frieze admits at most one rescaling to another (integral) \SL2-frieze, and the factor of any such rescaling is one of $\tfrac13$, $\tfrac12$, $2$, or $3$.
\end{lemma}
\begin{proof}
This observation follows immediately from the well-known and easily established fact (see, for example, \cite[Lemma 7.5]{CuHo2019}) that every quiddity of a positive width \SL2-frieze contains one of the substrings $(1, 2)$, $(2, 1)$, or $(1,3,1)$.
\end{proof}

\begin{proof}[Proof of the remaining part of Theorem~\ref{theorem:surj}]
It is straightforward to verify that any two \SL2-friezes that are mapped to the same Y-frieze by $p_n$ are necessarily related by a rescaling, and the factors of rescaling in two possible directions are mutually reciprocal. By Lemma~\ref{lemma:23}, there are at most two \SL2-friezes in the preimage of a Y-frieze under $p_n$, and if there are two, then one is the rescaling of the other by a factor of 2 or 3, as required.
\end{proof}

\begin{definition}
A triangulated polygon is \emph{2-multiflipative} if it can be partitioned into triangulated quadrilaterals such that the vertices around the triangulated polygon alternate between those that are incident to a diagonal of a quadrilateral and those that are not. A \emph{multiflip} of a 2-multiflipative triangulated polygon is a transformation that acts on the triangulation by preserving the quadrilaterals and replacing the diagonal within each quadrilateral by the other possible diagonal.

A triangulated polygon is \emph{3-multiflipative} if it can be partitioned into triangulated hexagons so that each hexagon has an internal triangle (that is, a triangle formed by three diagonals of the hexagon) and the vertices around the triangulated polygon alternate between those that are incident to a diagonal in a hexagon and those that are not. A \emph{multiflip} acts on this triangulation by preserving the hexagons and replacing the internal triangle within each hexagon by the other possible internal triangle.
See Figure \ref{fig:triangulations}.
\end{definition}

After a multiflip, a $\lambda$-multiflipative triangulated polygon remains as such, and a second multiflip restores the original triangulation.

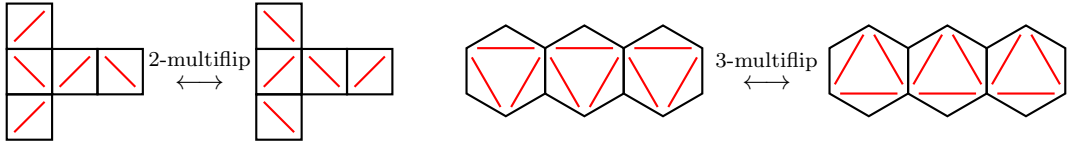
\begin{figure}[h!]
\centering
\begin{tikzpicture}[scale=0.6]

\begin{scope}[yshift=-0.5]

\draw[thick] (-1,0) rectangle (0,1);   
\draw[thick] (0,0) rectangle (1,1);    
\draw[thick] (1,0) rectangle (2,1);    
\draw[thick] (-1,1) rectangle (0,2);    
\draw[thick] (-1,-1) rectangle (0,0);   

\draw[red, thick, shorten >=4pt, shorten <=4pt] (-1,1) -- (0,2); 
\draw[red, thick, shorten >=4pt, shorten <=4pt] (1,1) -- (2,0); 
\draw[red, thick, shorten >=4pt, shorten <=4pt] (0,0) -- (1,1); 
\draw[red, thick, shorten >=4pt, shorten <=4pt] (0,0) -- (-1,1); 
\draw[red, thick, shorten >=4pt, shorten <=4pt] (0,0) -- (-1,-1); 

\end{scope}

\begin{scope}[yshift=0.5]
\node at (3.25,0.2) {$\longleftrightarrow$};
\node at (3.25,0.7) {\footnotesize 2-multiflip};
\end{scope}

\begin{scope}[xshift=5.5cm,yshift=-0.5]

\draw[thick] (-1,0) rectangle (0,1);   
\draw[thick] (0,0) rectangle (1,1);    
\draw[thick] (1,0) rectangle (2,1);    
\draw[thick] (-1,1) rectangle (0,2);    
\draw[thick] (-1,-1) rectangle (0,0);   

\draw[red, thick, shorten >=4pt, shorten <=4pt] (-1,2) -- (0,1); 
\draw[red, thick, shorten >=4pt, shorten <=4pt] (1,0) -- (2,1); 
\draw[red, thick, shorten >=4pt, shorten <=4pt] (0,1) -- (1,0); 
\draw[red, thick, shorten >=4pt, shorten <=4pt] (-1,0) -- (0,1); 
\draw[red, thick, shorten >=4pt, shorten <=4pt] (0,-1) -- (-1,0); 

\end{scope}

\begin{scope}[xshift=10cm,yshift=0.5cm]

\foreach \x in {0,1.7320508,3.4641016} {

  \coordinate (T)  at (\x+0, 1);
  \coordinate (TR) at (\x+0.8660254, 0.5);
  \coordinate (BR) at (\x+0.8660254,-0.5);
  \coordinate (B)  at (\x+0,-1);
  \coordinate (BL) at (\x-0.8660254,-0.5);
  \coordinate (TL) at (\x-0.8660254, 0.5);

  \draw[black, thick]
    (T)--(TR)--(BR)--(B)--(BL)--(TL)--cycle;

  \draw[red, thick, shorten >=4pt, shorten <=4pt] (TL)--(TR);
  \draw[red, thick, shorten >=4pt, shorten <=4pt] (TL)--(B);
  \draw[red, thick, shorten >=4pt, shorten <=4pt] (TR)--(B);

}

\node at (5.75,-0.3) {$\longleftrightarrow$};
\node at (5.75,0.2) {\footnotesize 3-multiflip};

\begin{scope}[xshift=8cm]

\foreach \x in {0,1.7320508,3.4641016} {

  \coordinate (T)  at (\x+0, 1);
  \coordinate (TR) at (\x+0.8660254, 0.5);
  \coordinate (BR) at (\x+0.8660254,-0.5);
  \coordinate (B)  at (\x+0,-1);
  \coordinate (BL) at (\x-0.8660254,-0.5);
  \coordinate (TL) at (\x-0.8660254, 0.5);

  \draw[black, thick]
    (T)--(TR)--(BR)--(B)--(BL)--(TL)--cycle;

  \draw[red, thick, shorten >=4pt, shorten <=4pt] (BL)--(BR);
  \draw[red, thick, shorten >=4pt, shorten <=4pt] (BL)--(T);
  \draw[red, thick, shorten >=4pt, shorten <=4pt] (BR)--(T);

}

\end{scope}

\end{scope}

\end{tikzpicture}
\caption{\label{fig:triangulations}
Two pairs of triangulated polygons related by multiflips}
\end{figure}

\begin{theorem}
\label{theorem:triangulations}
Two triangulated polygons are related by a multiflip if and only if their corresponding \SL2-friezes are related by a rescaling.
\end{theorem}
\begin{proof}
Suppose first that two triangulated polygon are related by a multiflip. The effect of a multiflip is to transform the quiddity $\sigma_i$ to $\lambda^{(-1)^{i}}\sigma_i$, for $\lambda\in\{\tfrac13,\tfrac12,2,3\}$. Since any \SL2-frieze is determined by its quiddity (using the diamond rule) we see that the two \SL2-friezes corresponding to the two triangulated polygons are related by a rescaling.

The more difficult part is the converse statement: if an \SL2-frieze admits rescaling by a factor of $\lambda$ to a different \SL2-frieze, the corresponding polygon is $\max(\lambda,1/\lambda)$-multiflipative.

Throughout the rest of the proof, we work with an \SL2-frieze $A$ that admits rescaling. In its quiddity, the number $\lambda=2$ or $3$ divides either every odd-indexed term or every even-indexed term. Colour the vertices of the triangulated polygon that correspond to the bisection of the quiddity that consists exclusively of multiples of $\lambda$ red, and the other vertices black. Now, as we go around the polygon, we encounter red and black vertices alternatingly.

Recall that the quiddity entry is equal to the number of triangles incident to the corresponding vertex in the corresponding triangulated polygon, while an arbitrary pair of vertices $(i,j)$ corresponds to the $A_{i,j}$.

Suppose $\lambda=2$. Then the quiddity of $A$ has a substring $(1,2)$ (or its reversion, which can be dealt with analogously). In the triangulated polygon, this corresponds to a black vertex incident to a single triangle and a red vertex incident to a pair of triangles, so they lie on a side of a quadrilateral in which the diagonal connects the red vertices and there is possibly another part of the polygon adjacent to the opposite side of the quadrilateral. Thus, we have partitioned the polygon into a quadrilateral and the remaining part, in which the vertices are again alternatingly red and black as we go around it, and the number of triangles incident to each red vertex is again even (because we removed two from one of them and did not touch the rest). Therefore, we can keep removing quadrilaterals from the remaining part until nothing is left, obtaining the desired partition.

\begin{figure}[h]
\centering
\begin{tikzpicture}

\draw[black, thick] (0,0) rectangle (1,1);
\draw[red, thick] (0.12,0.12) -- (0.88,0.88);

\draw[black] (0,1.1)--(0.25, 1.5) -- (0.75,1.5) -- (1,1.1) -- cycle;
\node at (0.5,1.25) {\ldots};

\fill[red]   (0,0) circle (1.5pt);
\fill[red]   (1,1) circle (1.5pt);
\fill[black] (1,0) circle (1.5pt);
\fill[black] (0,1) circle (1.5pt);

\fill[red]   (1,1.1) circle (1.5pt);
\fill[black] (0,1.1) circle (1.5pt);

\begin{scope}[shift={(4,0.5)}]

\coordinate (T)  at (90:1);
\coordinate (TR) at (30:1);
\coordinate (BR) at (-30:1);
\coordinate (B)  at (-90:1);
\coordinate (BL) at (-150:1);
\coordinate (TL) at (150:1);

\draw[black, thick]
  (T) -- (TR) -- (BR) -- (B) -- (BL) -- (TL) -- cycle;

\draw[black] (150:1.1) -- (150:1.4) -- (140:1.7) -- (130:1.7) -- (110:1.4) -- (95:1.08) -- cycle;
\node at (-0.8,1) {\ldots};
\draw[black] (30:1.1) -- (30:1.4) -- (40:1.7) -- (50:1.7) -- (70:1.4) -- (85:1.08) -- cycle;
\node at (0.8,1) {\ldots};

\draw[red, thick, shorten >=4pt, shorten <=4pt] (TL) -- (TR);
\draw[red, thick, shorten >=4pt, shorten <=4pt] (TL) -- (B);
\draw[red, thick, shorten >=4pt, shorten <=4pt] (TR) -- (B);

\fill[red]   (TL) circle (1.5pt);
\fill[red]   (TR) circle (1.5pt);
\fill[red]   (B)  circle (1.5pt);

\filldraw[fill=white, draw=black] (T)  circle (1.5pt);
\fill[black] (BR) circle (1.5pt);
\fill[black] (BL) circle (1.5pt);

\fill[red]  (150:1.1) circle (1.5pt);
\fill[red]  (30:1.1) circle (1.5pt);
\filldraw[fill=white, draw=black] (95:1.07) circle (1.5pt);
\filldraw[fill=white, draw=black] (85:1.07) circle (1.5pt);

\node at (0,-0.7) {$i$};
\node at (1.3,-0.5) {$i\!-\!1$};
\node at (0,1.33) {$j$};

\end{scope}

\end{tikzpicture}
\caption{Left: if $\lambda=2$, the triangulated polygon that corresponds the \SL2-frieze $A$ can be split into a quadrilateral and the remaining part, labelled by ellipsis. Right: if $\lambda=3$, the triangulated polygon that corresponds to $A$ can be split into a hexagon with an internal triangle and two more parts; vertex $j$ has to be black, even though it is not immediately obvious.}
\end{figure}
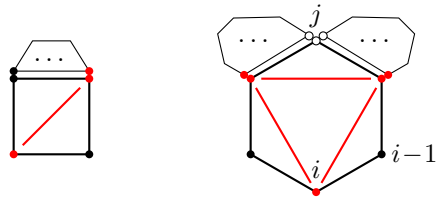

Now suppose that $\lambda=3$. Then its quiddity has a substring $(1,3,1)$. In the triangulated polygon, this corresponds to a red vertex $i$ incident to 3 triangles surrounded by two black vertices $i\pm1$ incident to one of these triangles each; moreover, the triangle in the middle has to have a third triangle adjacent to it, otherwise the polygon is just a pentagon which is never multiflipative. Thus, we have found a hexagon with an internal triangle with 4 consecutive sides being sides of the polygon. Suppose that the vertex $j$ isolated from these 4~sides is black. Then we have partitioned the polygon into a hexagon and two remaining parts, in each of which the vertices are again alternatingly red and black as we go around it, and the number of triangles incident to each red vertex is again a multiple of~3 (because we removed three from the vertices $i-2$ and $i+2$ and did not touch the rest). Therefore, we can keep removing hexagons until nothing is left, obtaining the desired partition.

It remains to show that the vertex $j$ is really black. Suppose otherwise. Note that $A_{i,j} = 2$ and, more importantly, $A_{i-1,j}=3$ (from any of the interpretations of \SL2-frieze entries given in \cite[Section 4.2]{Mo2015}). The $\lambda$-multiflip translates into the rescaling of the \SL2-frieze as follows: if $i,j$ are both red, $A_{i,j}$ gets multiplied by $\lambda$; if $i,j$ have different colours, $A_{i,j}$ remained unchanged; if $i,j$ are both black, $A_{i,j}$ gets divided by $\lambda$. Therefore, after the multiflip, in the resulting \SL2-frieze $A'$ we have $A'_{i,j} = 6$ and $A'_{i-1,j}=3$. But these two entries belong to a diamond in $A'$, making the left-hand side of the diamond rule a multiple of~3, which is impossible.
\end{proof}

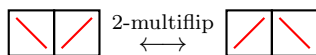
\begin{figure}[h]
\centering
\begin{tikzpicture}[scale=0.6]

\begin{scope}

\draw[thick] (-1,0) rectangle (0,1);   
\draw[thick] (0,0) rectangle (1,1);    

\draw[red, thick, shorten >=4pt, shorten <=4pt] (0,0) -- (1,1); 
\draw[red, thick, shorten >=4pt, shorten <=4pt] (0,0) -- (-1,1); 

\end{scope}

\begin{scope}[yshift=0.5]
\node at (2.4,0.2) {$\longleftrightarrow$};
\node at (2.4,0.7) {\footnotesize 2-multiflip};
\end{scope}

\begin{scope}[xshift=4.8cm]

\draw[thick] (-1,0) rectangle (0,1);   
\draw[thick] (0,0) rectangle (1,1);    

\draw[red, thick, shorten >=4pt, shorten <=4pt] (0,1) -- (1,0); 
\draw[red, thick, shorten >=4pt, shorten <=4pt] (-1,0) -- (0,1); 

\end{scope}

\end{tikzpicture}
\caption{The triangulated polygons that correspond to the \SL2-friezes from Examples \ref{ex:basic} and \ref{ex:other}, partitioned into quadrilaterals.}
\end{figure}

Note that for a given multiflipative triangulated polygon, its partition into quadrilaterals or hexagons is unique: it is formed precisely by the diagonals that connect vertices with different colours. In terms of the vertex colouring introduced in the proof, for the black vertices, the corresponding quiddity entries are equal to the number of quadrilaterals or hexagons incident to that vertex, and for the red vertices the number is $\lambda$ times that.

\section{Hyperbolic geometry of Y-friezes}
\label{sec:hyper}

\subsection{Distances between geodesics}
Theorems~\ref{theoremA} and~\ref{theoremC} have attractive interpretations in the hyperbolic plane, illustrated in Figures~\ref{figure1} and~\ref{fig:Ford}. Any simple closed clockwise path in the Farey graph can be represented by a chain of Ford circles $F_i$, where $F_{i}$ is the Ford circle centred at $a_i/b_i$. The $(i,j)$th entry of the corresponding $\text{SL}_2$-frieze is determined by the hyperbolic length between $F_i$ and $F_j$; specifically, $A_{i,j}=\exp \tfrac12 \varrho(F_i,F_j)$, where $\varrho$ denotes hyperbolic distance. This is known as the $\lambda$-length between $F_i$ and $F_j$; see \cite{Pe1987}. 

The same path can also be represented by the sequence of Farey edges $\gamma_i$, where $\gamma_i$ is the Farey edge from $a_{i}/b_{i}$ to $a_{i+1}/b_{i+1}$.
Given two hyperbolic lines $\gamma$ and $\gamma'$ with endpoints $p,q$ and $r,s$, we have $\sinh^2\tfrac12 \varrho (\gamma,\gamma')=[p,q,r,s]$, the cross-ratio, specified in Section~\ref{sec:intro}. Consequently, $B_{i,j}=\sinh^2\tfrac12 \varrho (\gamma_{i},\gamma_j)$. In this way we see that entries of $\text{SL}_2$-friezes encode distances between Ford circles and entries of Y-friezes encode distances between Farey edges.

\subsection{Multiflip and the Farey graph}
Using the geometry of $\mathscr{F}$, we can give a different proof of the combinatorial characterisation of multiflipative polygons.

\begin{proof}[The second proof of the converse statement of Theorem \ref{theorem:triangulations}]

Suppose that two \SL2-friezes $A$ and $A'$ are related by a rescaling, and assume for now that the scaling factors are 2 and $\tfrac12$. The quiddity $\sigma_i$ of one of these \SL2-friezes (say $A$) has $\sigma_i$ even for $i$ even. Let $a_i/b_i$ be a simple closed clockwise path in $\mathscr{F}$ of length $n+3$ satisfying $a_{i-1}b_i-b_{i-1}a_i=1$ and $A_{i,j}=a_jb_i-b_ja_i$. Then $a_{i+1}+a_{i-1}=\sigma_ia_i$ and $b_{i+1}+b_{i-1}=\sigma_ib_i$, for $i\in\mathbb{Z}$ (by, for example, \cite[Lemma~3.2]{ShVaZa2025}). Consequently, $a_{i+1}$ and $a_{i-1}$ have the same parity for $i$ even, as do $b_{i+1}$ and $b_{i-1}$. By applying an element of $\text{SL}_2(\mathbb{Z})$ we can assume that $a_i$ is even and $b_i$ is odd, for all odd $i$. 

Now, removing from $\mathscr{F}$ all odd/odd to odd/even edges gives a tessellation $\mathscr{F}^*$ by ideal quadrilaterals. The path $a_i/b_i$ lies in the 1-skeleton of $\mathscr{F}^*$, so it encloses an ideal polygon partitioned into ideal quadrilaterals. Reinserting the odd/odd to odd/even edges gives a 2-multiflipative triangulated polygon for which only vertices $a_i/b_i$ with $i$ even are incident to diagonals of quadrilaterals. Similarly, the triangulated polygon corresponding to $A'$ is 2-multiflipative, and from the quiddities we see that one triangulated polygon can be obtained from the other by a multiflip.

We can reason in a similar way when $A$ and $A'$ are related by scaling factors of 3 and $\tfrac13$. In this case, $\sigma_i$ is divisible by 3 for even $i$, and we can assume that $a_i$ is divisible by 3 and $b_i$ alternates between $\pm 1 \pmod{3}$ for odd $i$. We then remove from $\mathscr{F}$ all edges for which neither endpoint is congruent to $\pm (0,1) \pmod{3}$ to give a tessellation by ideal hexagons. Reinserting the removed edges gives a 3-multipflipative polygon, and the argument finishes as before.
\end{proof}

The geometry of $\lambda$-multiflipative triangulated polygons is related to the Hecke congruence subgroup of level $\lambda$, $\Gamma_0(\lambda) = \left\{\begin{psmallmatrix}a&b\\c&d\end{psmallmatrix}\in\SL2(\Z)\colon\ c\equiv0\pmod \lambda\right\}$. Specifically, if a 2-multiflipative (and therefore partitioned into triangulated quadrilaterals) triangulated polygon is embedded into $\mathscr{F}$, all ideal quadrilaterals from the partition lie in the same $\Gamma_0(2)$-orbit; if a 3-multiflipative (and therefore partitioned into hexagons with internal triangles) triangulated polygon is embedded into $\mathscr{F}$, all ideal hexagons from the partition lie in the same $\Gamma_0(3)$-orbit.

\subsection{Multiflip of Ford circle arrangement}
\begin{definition}
For a given horocycle centred at a real number in the upper half-plane model of the hyperbolic plane, its \emph{rescaling} by a factor of $\lambda\in\mathbb{R}_{>0}$ is an operation that produces another horocycle with the same centres but the (Euclidean) radius multiplied by $\lambda$.
For a closed chain of tangent horocycles $(H_0,H_1,\ldots,H_{N-1})$ with $N$ even, its \emph{rescaling} by a factor of $\lambda$ is rescaling each horocycle $H_i$ by a factor of $\lambda^{(-1)^i}$.
\end{definition}
Note that when a pair of tangent horocycles is simultaneously rescaled by reciprocal factors, they remain tangent, so a rescaled closed chain of tangent horocycles is also a closed chain of tangent horocycles. The process of rescaling a chain of horocycles in this way commutes with the usual action of real M\"obius transformations on the upper half-plane.

Since an \SL2-frieze corresponds to an entire $\SL2(\Z)$-class of closed chains of Ford circles, we can pick the single representative, called \emph{normalised} in \cite{MoOvTa2015}, that places the 0th Ford circle (and the vertex of the corresponding triangulated polygon in $\mathscr{F}$) at the point $\infty$ and the $(N-1)$th Ford circle at $0$.

\begin{theorem}
\label{prop:norm}
Suppose a $k$-multiflipative triangulated polygon is embedded in $\mathscr{F}$ as a normalised closed path $\gamma$. Let $\lambda=k$ if the vertex $1/0$ is incident to internal diagonals of polygons from the partition provided by multiflipativity, or $\lambda=1/k$ otherwise. Then the vertices of the normalised embedding into $\mathscr{F}$ of the triangulated polygon after a multiflip are the images of the vertices of $\gamma$ under the M\"obius transformation $z\mapsto\lambda z$.
\end{theorem}
\begin{proof}
The sequence of reduced fractions $\gamma=(1/0=a_0/b_0,a_1/b_1,\dots,a_{N-1}/b_{N-1}=0/1)$ can be read directly from two diagonals of the \SL2-frieze $A$ as $a_i=A_{i,-1}$, $b_i=A_{i,0}$ \cite{MoOvTa2015}. Multiflip in the triangulated polygon induces rescaling of the \SL2-frieze by a factor of $\lambda$, which, in turn, acts on the sequence of reduced fractions as $\gamma\mapsto\left( \tfrac10, \tfrac{\lambda a_1}{b_1}, \tfrac{a_2}{b_2/\lambda}, \tfrac{\lambda a_3}{b_3},\dots,\tfrac{a_{N-2}}{b_{N-2}/\lambda},\tfrac01\right)$.
\end{proof}

\begin{figure}[ht!]
\centering
\begin{tikzpicture}[scale=1.8]

\begin{scope}[xshift=-3cm]
\discfarey

\def\vertices{{{1,0},{4,1},{3,1},{2,1},{1,1},{0,1},{1,0}}}

\foreach \i in {0,...,5} {
    \pgfmathsetmacro{\numA}{\vertices[\i][0]}
    \pgfmathsetmacro{\denA}{\vertices[\i][1]}	
     \dischorocycle[black,fill=IndianRed,fill opacity=0.5](\numA:\denA);
}

\foreach \i in {0,...,5} {
    \pgfmathsetmacro{\numA}{\vertices[\i][0]}
    \pgfmathsetmacro{\denA}{\vertices[\i][1]}

    \pgfmathtruncatemacro{\nexti}{\i+1}
    \pgfmathsetmacro{\numB}{\vertices[\nexti][0]}
    \pgfmathsetmacro{\denB}{\vertices[\nexti][1]}

     \shline[ultra thick](\numA:\denA:\numB:\denB); 
}
\end{scope}

\begin{scope}
\discfarey

\def\vertices{{{sqrt(2),0},{4/sqrt(2),1/sqrt(2)},{3*sqrt(2),sqrt(2)},{2/sqrt(2),1/sqrt(2)},{sqrt(2),sqrt(2)},{0,1/sqrt(2)},{sqrt(2),0}}}

\foreach \i in {0,...,5} {
    \pgfmathsetmacro{\numA}{\vertices[\i][0]}
    \pgfmathsetmacro{\denA}{\vertices[\i][1]}	
     \dischorocyclenolabel[black,fill=IndianRed,fill opacity=0.5](\numA:\denA);
}

\end{scope}

\begin{scope}[xshift=3cm]
\discfarey

\def\vertices{{{1,0},{2,1},{3,2},{1,1},{1,2},{0,1},{1,0}}}

\foreach \i in {0,...,5} {
    \pgfmathsetmacro{\numA}{\vertices[\i][0]}
    \pgfmathsetmacro{\denA}{\vertices[\i][1]}	
     \dischorocycle[black,fill=IndianRed,fill opacity=0.5](\numA:\denA);
}

\foreach \i in {0,...,5} {
    \pgfmathsetmacro{\numA}{\vertices[\i][0]}
    \pgfmathsetmacro{\denA}{\vertices[\i][1]}

    \pgfmathtruncatemacro{\nexti}{\i+1}
    \pgfmathsetmacro{\numB}{\vertices[\nexti][0]}
    \pgfmathsetmacro{\denB}{\vertices[\nexti][1]}

     \shline[ultra thick](\numA:\denA:\numB:\denB); 
}
\end{scope}

\end{tikzpicture}
\caption{\label{fig:Ford}
Left: Farey path and chain of Ford circles that correspond to Example \ref{ex:basic}. Centre: horocycles after rescaling by a factor of~$\tfrac12$. Right: after M\"obius transformation $z\mapsto\tfrac12z$, we get the chain of Ford circles that corresponds to Example~\ref{ex:other}.}
\end{figure}
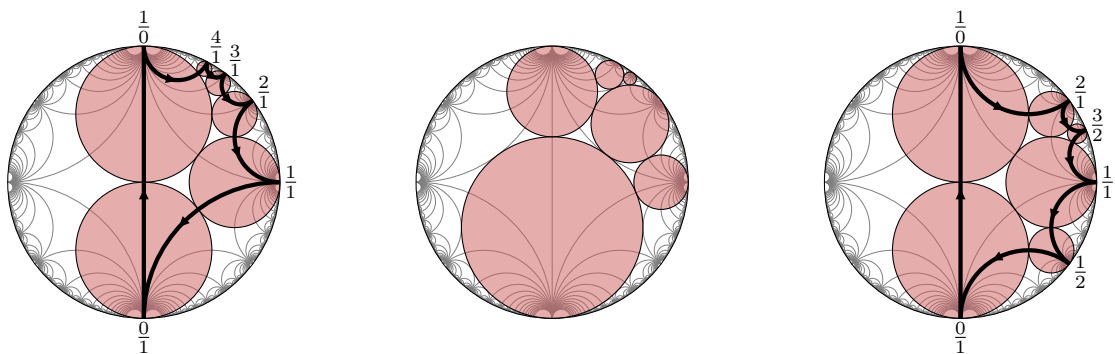

\begin{corollary}
Suppose a normalised closed chain of Ford circles corresponds to a $k$-multiflipative triangulated polygon. Let $\lambda=k$ if the vertex $\infty$ is incident to internal diagonals of polygons from the partition provided by multiflipativity, or $\lambda=1/k$ otherwise. Then rescaling the chain by the factor of ${\lambda}$ and then applying the M\"obius transformation $z\mapsto \lambda z$ gives the normalised chain of Ford circles that corresponds to the multiflipped polygon.
\end{corollary}
\begin{proof}
By Theorem~\ref{prop:norm}, the centres of the horocycles produced through rescaling (which does not affect the centres) and the M\"obius transofrmation are indeed the same as the centres of the desired configuration of the Ford circles. Moreover, the horocycle centred at 0 has the correct Euclidean diameter 1 after rescaling and the M\"obius transformation, so it is a Ford circle, and therefore all other horocycles in the chain are also Ford circles, as required.
\end{proof}

\section*{Acknowledgements}
We are grateful to Antoine de Saint Germain for useful discussions, including a sketch of the description of multiflips.
This research is supported by EPSRC grants EP/W002817/1 (IS) and EP/W524098/1 (AZ). This work benefited from discussions at the ICMS workshop \textit{Farey's legacy in frieze patterns and discrete geometry}.
		
\noindent \textit{Data availability:} there is no data associated with this article.

\end{document}